\documentclass[11pt,a4paper]{article}

\usepackage{amsmath, amsfonts, amssymb, amsthm}
\usepackage{bm}
\usepackage{autobreak} 

\usepackage{ulem}
\usepackage{mathrsfs}
\usepackage{geometry}
\usepackage{galois}
\usepackage{graphicx}
\usepackage[backref]{hyperref}
\usepackage{authblk}

\usepackage{soul}
\usepackage{color, xcolor}

\usepackage{comment}

\numberwithin{equation}{section} 
\newtheorem{thm}{Theorem}[section] 
\newtheorem{Def}[thm]{Definition}
\newtheorem{prop}[thm]{Proposition}
\newtheorem{lemma}[thm]{Lemma}
\newtheorem{coro}[thm]{Corollary}

\newtheorem{assum}[thm]{Assumption}
\newtheorem{remark}[thm]{Remark}

\newcommand\VV{H_0^1(\Omega)}
\newcommand\HH{L^2(\Omega)}
\newcommand\effi{(\|\nabla u\|^2+\|u_t\|^2)}
\newcommand\effip{(\|\nabla u\|^p+\|u_t\|^p)}
\newcommand\effit{(\|\nabla u(t)\|^2+\|u_t(t)\|^2)}
\newcommand\effitp{(\|\nabla u(t)\|^p+\|u_t(t)\|^p)}

\newcommand\effitaup{(\|\nabla u(\tau)\|^p+\|u_t(\tau)\|^p)}
\newcommand\intt{\int_0^t}
\newcommand\into{\int_{\Omega}}
\newcommand\intto{\int_0^t\int_{\Omega}}
\newcommand\trm[1]{\text{\rm #1}}

\title{Finite-dimensionality of attractors for wave equations with degenerate nonlocal damping}
\author{Zhijun Tang, Senlin Yan, Yao Xu\thanks{Corresponding author.} ~and Chengkui Zhong} 
\date{}

\begin{document}
    
    \maketitle

    \begin{abstract}
      In this paper we study the fractal dimension of global attractors for a class of wave equations with (single-point) degenerate nonlocal damping. Both the equation and its linearization degenerate into linear wave equations at the degenerate point and the usual approaches to bound the dimension of the entirety of attractors do not work directly. Instead, we develop a new process concerning the dimension near the degenerate point individually and show the finite dimensionality of the attractor.\\

\noindent \textbf{Keywords}: wave equations, degenerate nonlocal damping, global attractors, fractal dimension.
    \end{abstract}

    \section{Introduction}\label{sec_introduction}
    
    In this paper, we consider the dimension problem of the global attractor for the initial-boundary value problem
    \begin{equation}\label{Problem}
        \begin{cases}
            u_{tt} - \Delta u + (\|\nabla u\|^p + \|u_t\|^p)u_t + f(u) = 0 &\text{in } \Omega\times\mathbb{R}^+, \\
            u(x,0) = u_0(x), u_t(x,0) = u_1(x) &\text{in } \Omega, \\
            u(x,t) = 0 &\text{on } \partial\Omega, 
        \end{cases}
    \end{equation}
    where $p\in (1,2)$, $\Omega$ is a bounded domain in $\mathbb{R}^3$ with smooth boundary $\partial\Omega$ and $f$ satisfies the elementary Assumption \ref{f_assum1} below.
    \begin{assum}\label{f_assum1}
        $f\in C^2(\mathbb{R})$ and satisfies the critical growth condition
    \begin{equation}\label{f_cond1}
        |f''(s)|\leq C(1+|s|),
    \end{equation}
    and the dissipative condition
    \begin{equation}\label{f_cond2}
        \liminf_{|s|\rightarrow\infty} f'(s) \equiv \mu > -\lambda_1,
    \end{equation}
    where $\lambda_1$ is the first eigenvalue of $A$, the negative Laplacian operator with homogeneous Dirichlet boundary condition.
    \end{assum}

    There are massive works on asymptotic behaviours of solutions of nonlinear wave equations since 1970s. People are primarily concerned about the topics of well-posedness of the equations, and the existence, attracting speed as well as fractal dimension of the corresponding attractors. Models with various damping terms have been investigated, including weak damping $ku_t$, strong damping $-k\Delta u_t$, fractional damping $(-\Delta)^{\alpha}u_t$ with $\alpha\in(0,1)$ and nonlinear damping $g(u_t)$. For more information we refer readers to literature \cite{MR2026182,MR1868930,MR2018135,  MR1112054,MR2116726,  MR3504011,MR3561957,MR2237675,  MR2577597,MR1239923} and references therein.

    The dynamics of equations with nonlocal damping also attracted widespread attention. Authors in \cite{MR3986205} have studied well-posedness of a class of extensible beam models with nonlocal energy damping $\left(\|\Delta u\|^2+\|u_t\|^2\right)^q\Delta u_t$, which was first proposed by Balakrishnan-Taylor\cite{balakrishnan1989distributed}. Y. Sun and Z. Yang in \cite{MR4430610} have investigated the existence of strong global and exponential attractors for the equation
    \begin{equation*}
        u_{tt} + \Delta^2 u - \kappa\phi(\|\nabla u\|^2)\Delta u - M(\|\Delta u\|^2+\|u_t\|^2)\Delta u_t + f(u) = h,
    \end{equation*} 
    in which they assume that $M$ is positive on $\mathbb{R}^+$. We also refer to \cite{MR3642020,MR2853537,MR2771816} for more information about wave equations with damping term of the form $M(\|\nabla u\|^2_{\HH})(-\Delta)^{\theta}u_t$ and $M(\|\nabla u\|^2_{\HH})g(u_t)$, where the intensity of the damping involves only the potential part of the energy.

    In this paper we are concerned about the dissipative wave equation \eqref{Problem} with nonlocal damping $\effip u_t$, which is an extended Krasovskii model first studied by Chueshov\cite{MR3408002}. It has weaker damping and the energy decays more slowly near the origin than the equations with usual weak damping. Particularly, degeneration happens at the origin, which from the geometrical perspective complicates the asymptotic behaviours intrinsically. It is therefore interesting to ask if the fractal dimension of the global attractor is still finite, since as will be mentioned below the classical methods to bound fractal dimensions fail in our setting and in most cases degeneration causes infinite-dimensionality. Our main result is as follows.

\begin{thm}\label{sec_introduction_thm_main}
  Suppose Assumption \ref{f_assum1} holds and in addition
  \begin{equation}\label{f_cond3}
    f(0) = 0, \quad f'(0)> -\big(1-\frac{p}{4}\big)\lambda_1,\quad|f'(s)-f'(0)| = o(|s|^p),
  \end{equation}
Then the global attractor of Problem \eqref{Problem} in $\VV\times\HH$ has finite fractal dimension.
\end{thm}

We should mention that condition \eqref{f_cond3} is only involved in the proof of Lemma \ref{Lipschitz_u} which declares the uniform Lipschitz stability of the semigroup near the origin. However, to prove Theorem \ref{sec_introduction_thm_main} only the uniform H\"{o}lder property on the attractor near the origin is required. Therefore, condition \eqref{f_cond3} shall not be sharp. Especially, by calculating \eqref{diff_estimate_utUt} more accurately, one can improve $\frac{p}{4}$ in the condition $f'(0)> -\big(1-\frac{p}{4}\big)\lambda_1$ with a slightly smaller number $C_p$.

    Dimension problem is an essential subject in the area of infinite-dimensional dynamical systems, which may reflect the complexity of the system. If the fractal dimension of the global attractor is finite, in some sense one can reduce the dynamics on the infinite-dimensional phase space to a simpler one on a finite-dimensional phase space. There have been already some classical approaches to bound the (finite) fractal dimension specifically for negatively invariant sets, including global attractors. For instance, one can use the method of Lyapunov exponent in \cite{MR1441312,MR2767108} to study finite-dimensional behaviors if the quasi-derivative of $S(T)$ is a compact perturbation of some uniform contractive operator for some fixed $T>0$. Another alternative is the quasi-stability inequality method if the semigroup $S(t)$, not necessarily differentiable, satisfies
    \begin{equation}\label{quasi-stable}
        \|S(T)v_1-S(T)v_2\| \leq \eta \|v_1-v_2\| + n_Z(Kv_1-Kv_2),
    \end{equation}
    where $T>0, 0<\eta<1$, $K$ is a Lipschitz mapping from the attractor $\mathscr{A}$ to a Banach space $Z$ and $n_Z$ is a compact seminorm on $Z$. All these methods require some kind of uniform contraction property on the entirety of the attractors. We refer readers to \cite{MR3408002} for more details and to \cite{MR2870874} for exponential attractors of semigroups on Banach spaces.

    However, the existing methods all lose efficacy for our problem due to the degeneration occurring at the origin. In fact, the semigroup behaves just like the wave equation
    \begin{equation*}
        u_{tt} - \Delta u= 0
    \end{equation*}
    near the origin if $f(0)=0$, in which situation the energy does not decay along the direction of each eigenfunction of $A$. Therefore, we can neither write the derivative of $S(T)$ into a compact perturbation of some uniform contractive operator, nor find a constant $\eta<1$ to fulfill \eqref{quasi-stable} uniformly (see Remark \eqref{remark_reason} for details).

    It is worth pointing out that there are just few works dealing with the dimension problem of degenerate models, and almost all of them possess infinite-dimensional attractors. In \cite{MR2392315} M. Efendiev and S. Zelik have proved the finite-dimensionality of global attractors for a class of porous medium equations when the nonlinearity $g(u)$ satisfies $g'(0)>0$. Moreover, they also have obtained the infinite-dimensionality in the case $g'(0)<0$ by estimating Kolmogorov $\epsilon$-entropy. See also \cite{MR2372422,MR2823885} for the corresponding infinite-dimensionality results on parabolic equations with $p$-Laplacian. Besides, using another method of the $Z_2$ index, C. Zhong and W. Niu have also shown in \cite{MR2728546} that the fractal dimension of the global attractor is infinite for a class of $p$-Laplacian equations.

    We now explain the main difficulties and new ideas in the proof while describing the outline of this paper. In Section \ref{sec_well-posedness}, we establish the well-posedness of strong solutions and generalized solutions to problem \eqref{Problem} by the monotone operator method under Assumption \ref{f_assum1}. In the same section, the existence of the global attractor $\mathscr{A}$ in $H^1_0(\Omega)\times L^2(\Omega)$ is proved through the dissipation and asymptotic smoothness. To verify the asymptotic smoothness, we make a decomposition $u=v+w$, such that, $v$ decays uniformly to zero and $w$ possesses higher regularity for each $t>0$ for bounded initial data. This special decomposition is also valid for establishing higher regularity of the global attractor under additional assumption \ref{assum_2} on $f$ in the next section, which infers directly, as a byproduct, the temporal H\"{o}lder continuity of $S(t)$ uniformly on $\mathscr{A}$. More detailedly, the regularity of $\mathscr{A}$ can be deduced from the higher regularity of the attracting set of $w$ and the latter conclusion can be achieved by a process of weighted energy estimate, inspired by the structure of the damping, together with an extended discrete-like Gronwall lemma. This is a novelty of the paper. Another essential estimate in Section \ref{sec_regularity-decay} is Lemma \ref{u_decay_equiv} concerning the accurate decay rate of solutions near the origin on the phase space. Compared to the usual equations with nondegenerate weak damping whose solutions decay exponentially fast near the origin, in our problem the rates of both the norm and the energy are as slow as $t^{-\frac{1}{p}}$. These inherent estimates obtained in Section \ref{sec_regularity-decay} are useful to perform the dimension estimation in the last section.

    Section \ref{sec_dimensions} is devoted to an abstract framework of the dimension calculation for single-point degenerate equations. In consideration of problem \eqref{Problem}, we employ the setting that the degenerate point $0$ is a locally attracting point with a quantitative speed on the attractor and the semigroup is uniformly (in time and space) H\"{o}lder continuous near the origin. Intuitively, the attracting behaviour near $0$ plays a role of contraction effect as the trajectories collapse to $0$. Therefore, it is reasonable to believe the geometry of the attractor near $0$ would not be worse than that of the region away from $0$. Indeed, providing that the dimension of the complement on the attractor of neighbourhoods of $0$ is finite, we are able to declare the finite dimensionality of neighbourhoods of $0$, as stated in Theorems \ref{main} and \ref{main_2}. In this process we decompose the entire attractor into the degenerate part (the neighbourhood of $0$) and the nondegenerate part (the complement of the neighbourhood). The former set is positively invariant and the latter one is negatively invariant. Hence, it is possible to figure out the dimension of the latter one through the pioneering classical methods while the dimension problem of the former part has been reduced to that of the latter part. This is the central idea in Section \ref{sec_dimensions}.

    In Section \ref{sec_application}, we apply the previous theory to problem \eqref{Problem}. To do this, we first show the uniform Lipschitz stability of $S(t)$ near the origin in Lemma \ref{Lipschitz_u} by estimating the operator norm of the derivative $DS(t)$. To calculate the dimension of the nondegenerate part, we split $DS(t)$ into a uniform contractive linear operator and a compact operator, which gives rise to the finite-dimensionality of this negatively invariant set. As a result, Theorem \ref{main_2} leads to Theorem \ref{sec_introduction_thm_main} immediately. 

    Throughout the paper, we use the notations $\|\cdot\|_p=\|\cdot\|_{L^p(\Omega)}$ (particularly, $\|\cdot\|=\|\cdot\|_2$ for brevity) and, for a function $u$
    \begin{equation}\label{I_u_definition}
        I_u(t) = \|\nabla u\|^2 + \|u_t\|^2, \quad I_{u,p}(t) = \|\nabla u\|^p + \|u_t\|^p.
    \end{equation}
    Without causing misunderstanding, we also write both the inner product in $\HH$ and the ordered pair of $u, v$ as $(u,v)$. Besides, we may use $C$ to denote any positive constant which may differ from each other, even in the same line.

    \section{well-posedness and global attractor}\label{sec_well-posedness}
    In this section we investigate the global well-posedness of Problem \eqref{Problem} as well as the existence of global attractors under Assumption \ref{f_assum1}. 
    \begin{Def}
        Suppose $u\in C([0,T];H^1_0(\Omega)) \cap C^1([0,T];L^2(\Omega))$ with $u(0)=u_0$ and $u_t(0)=u_1$. Then we say
        \item[(i)] $u$ is a strong solution if
        \begin{itemize}
            \item [a)] $u\in W^{1,1}(a,b;H_0^1(\Omega))$ and $u_t\in W^{1,1}(a,b;L^2(\Omega)), \forall\, 0<a<b<T$;
            \item [b)] $-\Delta u(t) \in L^2(\Omega), \forall\, t\in [0,T]$;
            \item [c)] \eqref{Problem} holds in $L^2(\Omega)$ for almost every $t\in [0,T]$.
        \end{itemize}
        \item[(ii)] $u$ is a generalized solution if there exists a sequence of strong solutions $\{ u^{(j)} \}$ with initial data $(u^{(j)}_0,u^{(j)}_1)$ such that
        $$
        (u^{(j)},u_t^{(j)}) \rightarrow (u,u_t) \text{ in } C([0,T];H_0^1(\Omega)\times L^2(\Omega)).
        $$
        \item[(iii)] $u$ is a weak solution if
        \begin{align}\label{weak_def}
          \begin{split}
            &\int_{\Omega} u_t(x)\psi(x) dx + \int_0^t \Big[ \int_{\Omega} \nabla u(\tau,x)\nabla \psi(x) dx +  \int_{\Omega} f(u(\tau,x))\psi(x) dx \\ 
            &\qquad + (\|\nabla u(\tau)\|^p + \|u_t(\tau)\|^p) \int_{\Omega}u_t(\tau,x)\psi(x) dx \Big] d\tau = \int_{\Omega} u_1(x)\psi(x) dx
          \end{split}
        \end{align}
        holds for all $\psi\in H_0^1(\Omega)$ and almost every $t\in [0,T]$.
      \end{Def}
      
    \begin{thm}\label{well-posed}
        Under Assumption \ref{f_assum1}, we have the following conclusions:
        \item[(i)] If $(u_0,u_1)\in (H^2(\Omega)\cap\VV)\times\VV$ , then Problem \eqref{Problem} admits a unique strong solution $u$, which satisfies that
            \begin{align*}
                &(u_t,u_{tt})\in L^{\infty}(0,T;\VV\times\HH), \\
                &u_t\in C_r([0,T),\VV),\quad u_{tt}\in C_r([0,T),\HH), \\
                &-\Delta u + \effi u_t \in C_r([0,T),\HH),
            \end{align*}
        where $C_r$ is the space of right-continuous functions.
        \item[(ii)] If $(u_0,u_1)\in\VV\times\HH$, Problem \eqref{Problem} admits a unique generalized solution, which is also a weak solution. Furthermore, $u$ satisfies the energy equality
        \begin{equation}\label{energy_eq}
            E(t) + \int_0^t \effitaup \|u_t(\tau)\|^2 d\tau = E(0),
        \end{equation}
        where $E(t) \triangleq \frac{1}{2}\effit + \int_{\Omega} F(u(t)) dx$ with $F(s) \triangleq \int_0^s f(\tau) d\tau$.
\end{thm}

    To prove Theorem \ref{well-posed}, we use the monotone operator method as in \cite{MR1941662}. 

    \begin{lemma}\label{sol_monotone_op}
        Let $A:D(A)\subset H\to H$ is a maximal monotone operator on a Hilbert space $H$, i.e., $(Ax_1-Ax_2,x_1-x_2)_H\geq 0$ for any $x_1,x_2\in D(A)$ and $Rg(I+A)=H$; besides, assume that $0\in A0$. Let $B:H\to H$ be locally Lipschitz. If $u_0\in D(A),f\in W^{1, 1}(0,t;H)$ for all $t>0$, then there exists $t_{max}\leq \infty$ such that the initial value problem
        \begin{equation}\label{eq_monotone_op}
            u_t+Au+Bu \ni f \text{ and } u=u_0
        \end{equation}
        has a unique strong solution $u$ on the interval $[0,t_{max})$.
        
        Whereas, if $u_0\in \overline{D(A)}, f\in L^1(0,t;H)$ for all $t>0$, then problem \eqref{eq_monotone_op} has a unique generalized solution $u\in C([0,t_{max});H)$.
        
        Moreover, in both case we have $\lim_{t\to t_{max}} \|u(t)\|_H=\infty$ provided $t_{max}<\infty$.
    \end{lemma}

\begin{proof}[Proof of Theorem \ref{well-posed}]
  \textbf{Step 1.} We establish the existence and uniqueness of strong solutions and generalized solutions.

  Let $U=(u,v)^T$, where $v=u_t$. Define $\mathcal{A}:D(\mathcal{A})\subset\VV\times\HH \rightarrow \VV\times\HH$ and $\mathcal{B}:\VV\times\HH\rightarrow\VV\times\HH$ by
    \begin{align*}
        \mathcal{A}U=\left(
      \begin{array}{c}
        -v \\
        -\Delta u
      \end{array}
      \right),\quad \mathcal{B}U=\left(
      \begin{array}{c}
        0 \\
        f(u) + (\|v\|^p+ \|\nabla u\|^p)v
      \end{array}
      \right),
    \end{align*}
    respectively, where $D(\mathcal{A}) = (H^2(\Omega)\cap\VV)\times\VV$. Then Problem \eqref{Problem} can be written into
    \begin{equation*}
            \begin{cases}
                U_t + \mathcal{A}U + \mathcal{B}U = 0, \\
                U(0) = U_0,
            \end{cases}
    \end{equation*}
    where $U_0=(u_0,u_1)^T$. It is easy to see that $\mathcal{A}$ is monotone. To prove $\mathcal{A}$ is maximal monotone, it is sufficient to show that
    \begin{equation}\label{subjectivity}
        Rg(\mathcal{A}+I) = \VV\times\HH.
    \end{equation}
    Indeed, for any $(f_0,f_1)\in \VV\times\HH$, we consider the following equation
    \begin{equation}\label{equa_A+I}
        (\mathcal{A}+I)U = \left(
            \begin{aligned}
                -v+u \\
                -\Delta u+v
            \end{aligned}
        \right) = \left(
            \begin{aligned}
                f_0 \\ f_1
            \end{aligned}
        \right).
    \end{equation}
    Eliminating $u$ and $v$ respectively, we transform \eqref{equa_A+I} into 
    \begin{equation}\label{equa_G}
        \begin{cases}
            -\Delta u + u = f_0 + f_1 \in \HH, \\
            -\Delta v + v = \Delta f_0 + f_1 \in H^{-1}(\Omega).
        \end{cases}
    \end{equation}
    By the method of variations, it is well known that \eqref{equa_G} is well-posed, and so is \eqref{equa_A+I}. Therefore, we have proved \eqref{subjectivity} and $\mathcal{A}$ is maximal monotone. Denote $\theta_u = \theta u_1 + (1-\theta)u_2, \theta_v = \theta v_1 + (1-\theta)v_2$. It holds that
    \begin{align*}
      \|f(u_1)-f(u_2)\|^2
      &= \into \left|\int_0^1 f'(\theta_u)(u_1-u_2)d\theta \right|^2 dx \\
      &\leq C\into |u_1-u_2|^2 \left(\int_0^1(1+\theta_u^2)d\theta\right)^2dx \\
      &\leq C\|u_1-u_2\|_6^2 \left(\into (1+|u_1|^2+|u_2|^2)^3 dx\right)^{\frac{2}{3}} \\
      &\leq C(1+\|\nabla u_1\|+\|\nabla u_2\|)^4 \|\nabla u_1-\nabla u_2\|^2,
    \end{align*}
    i.e., $\|f(u_1)-f(u_2)\|\leq C(1+\|\nabla u_1\|+\|\nabla u_2\|)^2 \|\nabla u_1-\nabla u_2\|$. Besides, we have
    \begin{align*}
      \left\| \|v_1\|^p v_1 - \|v_2\|^p v_2 \right\|
      &\leq \left|\|v_1\|^p-\|v_2\|^p\right|\|v_1\| + \|v_2\|^p \|v_1-v_2\| \\
      &=\left|p\int_0^1\|\theta_v\|^{p-2}(\theta_v,v_1-v_2)d\theta\right|\|v_1\| + \|v_2\|^p\|v_1-v_2\| \\
      &\leq \left(C_p(\|v_1\|+\|v_2\|)^{p-1}\|v_1\|+\|v_2\|^p\right)\|v_1-v_2\|,
    \end{align*}
    and
    \begin{align*}
      &\| \|\nabla u_1\|^pv_1 - \|\nabla u_2\|^pv_2 \| 
      = \left\| (\|\nabla u_1\|^p-\|\nabla u_2\|^p)v_1 + \|\nabla u_2\|^p(v_1-v_2) \right\| \\
      \leq &\left|p\int_0^1\|\nabla \theta_u\|^{p-2}(\nabla \theta_u,\nabla u_1-\nabla u_2)d\theta\right|\|v_1\| + \|\nabla u_2\|^p \|v_1-v_2\| \\
      \leq &C_p\|v_1\|(\|\nabla u_1\|+\|\nabla u_2\|)^{p-1}\|\nabla u_1-\nabla u_2\| + \|\nabla u_2\|^p\|v_1-v_2\|.
  \end{align*}
    Therefore, $\mathcal{B}$ is locally Lipschitz on $\VV\times\HH$. Hence, it follows from Lemma \ref{sol_monotone_op} that, for any $(u_0,u_1)\in (H^2(\Omega)\cap\VV)\times\VV$, there exists $t_{max}\leq\infty$ such that Problem \eqref{Problem} has a unique strong solution on $[0, t_{max})$. Noticing that $D(\mathcal{A})$ is dense in $\VV\times\HH$, we also know that, for any $(u_0,u_1)\in \VV\times\HH$, Problem \eqref{Problem} has a unique generalized solution. Besides, both strong solutions and generalized solutions belong to $C([0,t_{max});\VV\times\HH)$ and $t_{max}$ is maximal in the sense that
    \begin{equation*}
        \text{if } t_{max}<\infty, \text{ then } \lim_{t\rightarrow t_{max}} \|(u,u_t)\|_{\VV\times\HH} = \infty.
    \end{equation*}

    To prove the global well-posedness, it suffices to show that $\|(u,u_t)\|_{\VV\times\HH}$ would not blow up in finite time. Taking $\mu_0\in (-\mu,\lambda_1)$, we know from \eqref{f_cond2} that there exists $M>0$ such that 
    \begin{equation}\label{df_inf}
        f'(s)>-\mu_0\quad\textrm{for }|s|>M,
    \end{equation}
which implies that
    \begin{equation*}
            \begin{cases}
                F(s)\geq -\frac{\mu_0}{2}s^2-C_1, &|s|>M, \\
                F(s)\leq C_1, &|s|\leq M.
            \end{cases}
    \end{equation*}
    Therefore, we have that
        \begin{align*}
            \into F(u)dx
            &= \int_{\{|u|>M\}} F(u) dx + \int_{\{|u|\leq M\}} F(u) dx \geq -\frac{\mu_0}{2} \|u\|^2 -C \\&\geq -\frac{\mu_0}{2\lambda_1} \|\nabla u\|^2 - C,
        \end{align*}
  and further
    \begin{equation}\label{E_inf}
        E(t)\geq \frac{1}{2}\Big(1-\frac{\mu_0}{\lambda_1}\Big)(\|\nabla u\|^2 + \|u_t\|^2) - C.
    \end{equation}
    Multiplying equation \eqref{Problem} by $u_t$ and integrating over $\Omega\times(0,t_{max})$, we know that $u$ satisfies the energy equality \eqref{energy_eq}, which implies $E(t)$ is non-increasing with $t$. This tells us that $E(t)$ would not blow up in finite time, and so is $\|(u, u_t)\|_{H^1_0(\Omega)\times L^2(\Omega)}$ by \eqref{E_inf}. Hence the strong solutions and generalized solutions are globally well-posed.

    \textbf{Step 2.} We prove that generalized solutions are also weak solutions. Let $(u,u_t)$ be a generalized solution with initial data $(u_0,u_1)$. Then there exists a sequence of strong solutions $(u^{(j)},u_t^{(j)})$ with initial data $(u_0^{(j)},u_1^{(j)})$ such that
    \begin{equation*}
        (u^{(j)},u_t^{(j)})\rightarrow (u,u_t) \text{ in } C([0,T];\VV\times\HH) \text{ as } j\rightarrow\infty.
    \end{equation*}
    Obviously, \eqref{weak_def} holds for each $(u^{(j)},u_t^{(j)})$, i.e.
        \begin{align*}
            &\int_{\Omega} u^{(j)}_t(x)\psi(x) dx + \int_0^t \left[ \int_{\Omega} \nabla u^{(j)}(\tau,x)\nabla \psi(x) dx + \int_{\Omega} f(u^{(j)}(\tau,x))\psi(x) dx \right. \\ 
            &+ \left.(\|\nabla u^{(j)}(\tau)\|^p + \|u^{(j)}_t(\tau)\|^p) \int_{\Omega}u^{(j)}_t(\tau,x)\psi(x) dx \right] d\tau = \int_{\Omega} u^{(j)}_1(x)\psi(x) dx.
        \end{align*}
      It is easy to see that
        \begin{align*}
          \intto\nabla u^{(j)}\nabla\psi+f(u^{(j)})\psi dxdt\rightarrow\intto \nabla u\nabla\psi+f(u)\psi dxdt.
    \end{align*}
Besides, since $\into u_t^{(j)}\psi dx$ converges to $\into u_t\psi dx$ for all $t\in [0,T]$ and $(\|\nabla u^{(j)}\|^p+\|u^{(j)}_t\|^p)(u_t^{(j)},\psi)$ is bounded on $[0,T]$, we infer from Lebesgue Dominated Convergence Theorem that
    \begin{equation*}
        \intt (\|\nabla u^{(j)}\|^p+\|u^{(j)}_t\|^p)\into u_t^{(j)}\psi dxdt \rightarrow \intt \effip\into u_t\psi dxdt.
    \end{equation*}
    Hence, equality \eqref{weak_def} holds for $u$, i.e., $u$ is also a weak solution. Following a similar argument, one can verify the energy equality \eqref{energy_eq}.
  \end{proof}
  
  By Theorem \ref{well-posed}, Problem \eqref{Problem} generates a semigroup on $\VV\times\HH$, denoted by $S(t)$. We will show $S(t)$ possesses a global attractor. Let us start with the dissipation.
  
    \begin{thm}\label{dissipativity}
        Under Assumption \ref{f_assum1}, the dynamical system $(\VV\times\HH, S(t))$ generated by Problem \eqref{Problem} is dissipative. In other words, there exists a constant $R>0$ such that for any bounded set $B\subset\VV\times\HH$, there exists $t_0=t_0(B)$ satisfying, for any $t\geq t_0$ and any $(u,v)\in B$,
    \begin{equation*}\label{diss_cond}
        \|S(t)(u,v)\|_{\VV\times\HH} \leq R.
    \end{equation*}
\end{thm}

\begin{proof}
  Denote $K= \sup_{(u,v)\in B} \|(u,v)\|_{\VV\times\HH} < \infty$. Multiplying equation \eqref{Problem} by $u_t + \epsilon u$ for some $\epsilon>0$ small and integrating over $\Omega$, we have
    \begin{align}\label{V_estimate}
        \begin{split}
            \frac{d}{dt}E_{\epsilon}(t) + \epsilon E_{\epsilon}(t) &+ \effitp\|u_t(t)\|^2 + \frac{\epsilon}{2}\|\nabla u(t)\|^2 \\
            = \frac{3\epsilon}{2}\|u_t(t)\|^2 &+\epsilon^2(u_t(t),u(t)) + \epsilon\into F(u(t))dx - \epsilon(f(u(t)),u(t)) \\
            &- \epsilon\effitp(u_t(t),u(t)),
        \end{split}
    \end{align}
    where $E_{\epsilon}(t) = \frac{1}{2}\effit + \int_{\Omega} F(u(t)) dx + \epsilon(u(t),u_t(t))$. By \eqref{E_inf} we know that if $\epsilon<\epsilon_1$ for some $\epsilon_1$ small
    \begin{align}
            E_{\epsilon}(t) 
            &= E(t) + \epsilon(u(t),u_t(t))\geq \frac{1}{2}\left(1-\frac{\mu_0}{\lambda_1}\right)(\|\nabla u\|^2 + \|u_t\|^2) - C + \epsilon(u(t),u_t(t)) \nonumber\\
            &\geq \frac{1}{4}\left(1-\frac{\mu_0}{\lambda_1}\right)(\|\nabla u\|^2 + \|u_t\|^2) - C.\label{V_inf}
    \end{align}
    Let us estimate the right hand side of \eqref{V_estimate} term by term. Firstly, it holds that
    \begin{equation}\label{V_estimate_ineq1}
        \frac{3\epsilon}{2}\|u_t(t)\|^2 \leq C_p\epsilon^{\frac{p+2}{p}} + \frac{1}{4}\|u_t\|^{p+2},
    \end{equation}
    and
    \begin{align}\label{V_estimate_ineq2}
        \begin{split}
            \epsilon^2|(u_t(t),u(t))| &\leq \lambda_1^{-\frac{1}{2}}\epsilon^2\|u_t\|\|\nabla u\| \leq \frac{\epsilon^2}{2\sqrt{\lambda_1}}(\|u_t\|^2\|\nabla u\|^p+\|\nabla u\|^{2-p}) \\
            &\leq \frac{\epsilon^2}{2\sqrt{\lambda_1}}(\|u_t\|^2\|\nabla u\|^p+\|\nabla u\|^2+C_p) \\
            &\leq \frac{1}{4}\|u_t\|^2\|\nabla u\|^p + \frac{\epsilon}{8}(1-\frac{\mu_0}{\lambda_1})\|\nabla u\|^2 + C_p\epsilon,
        \end{split}
    \end{align}
    if $\epsilon<\epsilon_2$ for some $\epsilon_2$ small. For $|s|\geq M$, we can infer from \eqref{df_inf} that
    \begin{align*}
            \left.\left( \frac{1}{2}\mu_0 s^2 + F(s) \right)\right|_M^s = \int_M^s(f(\tau)+\mu_0 \tau)d\tau \leq (f(s)+\mu_0 s)(s-M),
    \end{align*}
    which implies
    \begin{equation*}
        F(s)\leq f(s)s + \frac{\mu_0}{2}s^2 + C\quad \textrm{for all }|s|>M.
    \end{equation*}
    Hence,
    \begin{align}
            \epsilon\into F(u)dx - \epsilon(f(u),u) 
            &\leq \epsilon\Big(\int_{\{|u|>M\}} (F(u)-f(u)u)dx +\int_{\{|u|\leq M\}} (F(u)-f(u)u)dx\Big) \nonumber\\
            &\leq \epsilon\int_{\{|u|>M\}} \left(\frac{\mu_0}{2}|u|^2 + C\right) + \epsilon\int_{\Omega} C_M \nonumber\\
            &\leq \frac{\mu_0\epsilon}{2\lambda_1}\|\nabla u\|^2 + C\epsilon.\label{V_estimate_ineq3}
    \end{align}
Noticing that $E(t)$ is non-increasing, we know from \eqref{E_inf} that
    \begin{equation*}
        \|\nabla u(t)\|^2+\|u_t(t)\|^2 \leq C_K\textrm{ and } E_\epsilon(t)\leq C_K,\quad \textrm{for all } t\geq 0,
    \end{equation*}
    where $C_K>0$ is a constant depending on $K$.
    Setting $\epsilon<\epsilon_3=\min\{\frac{1}{4}\lambda_1^{\frac{1}{2}}, \frac{\lambda_1-\mu_0}{8C_K^p}\}$, it follows from Young's inequality that
    \begin{align}\label{V_estimate_ineq4}
        \begin{split}
            &\quad\epsilon\effitp(u_t(t),u(t))
            \leq \epsilon\lambda_1^{-\frac{1}{2}} \|\nabla u\|^{p+1}\|u_t\| + \epsilon\lambda_1^{-\frac{1}{2}} \|\nabla u\|\|u_t\|^{p+1} \\
            &\leq \frac{\epsilon}{8}\left(1-\frac{\mu_0}{\lambda_1}\right) \|\nabla u\|^2 + \frac{2\epsilon}{\lambda_1}\left(1-\frac{\mu_0}{\lambda_1}\right)^{-1} \|\nabla u\|^{2p}\|u_t\|^2 \\&\qquad\qquad+ \epsilon\lambda_1^{-\frac{1}{2}} \left(\frac{1}{p}\|\nabla u\|^p\|u_t\|^2 + \frac{p-1}{p}\|u_t\|^{p+2}\right) \\
            &\leq \frac{\epsilon}{8}\left(1-\frac{\mu_0}{\lambda_1}\right) \|\nabla u\|^2 + \frac{1}{2}\effip\|u_t\|^2.
        \end{split}
    \end{align}
    Inserting \eqref{V_estimate_ineq1}-\eqref{V_estimate_ineq4} into \eqref{V_estimate}, we conclude that
    \begin{equation*}
        \frac{d}{dt}E_{\epsilon}(t) + \epsilon E_{\epsilon}(t) + \frac{1}{4}\|\nabla u(t)\|^p\|u_t(t)\|^2 + \frac{1}{2}\|u_t\|^{p+2} + \frac{\epsilon}{4}\left(1-\frac{\mu_0}{\lambda_1}\right) \|\nabla u\|^2 \leq C_p\epsilon,
    \end{equation*}
    whenever $\epsilon<\min\{\epsilon_1,\epsilon_2,\epsilon_3\}$ is fixed. Hence, by the Gronwall lemma, we get
    \begin{equation*}
        E_{\epsilon}(t)\leq E_{\epsilon}(0)e^{-\epsilon t} + C(1-e^{-\epsilon t}) \leq C_Ke^{-\epsilon t} + C,
    \end{equation*}
which, by choosing $t_0 = \epsilon^{-1}\ln(C_K)$, implies
    \begin{equation*}
        E_{\epsilon}(t)\leq C+1\quad\textrm{for all } t\geq t_0.
    \end{equation*}
This, together with \eqref{V_inf}, yields the dissipation. 
\end{proof}
  
As mentioned in the introduction, we can not deal with the damping term directly since it is neither compact nor contractive. To overcome this difficulty, we would make a decomposition and prove the asymptotic smoothness based on Propositions \ref{asym_smooth} and \ref{attractor_ex}.
    \begin{prop}[\cite{MR3408002}]\label{asym_smooth}
        Let $S(t)$ be an evolution operator in a Banach space $X$. Assume that for each $t>0$ there exists a decomposition $S(t)=S^{(1)}(t)+S^{(2)}(t)$, where $S^{(2)}(t)$ is a mapping in $X$ satisfying that for any bounded set $B$ 
        \begin{equation}\label{S2_cond}
            r_B(t)=\sup\left\{\|S^{(2)}(t)x\|_X:x\in B\right\}\to 0 \text{ as } t\to\infty,
        \end{equation}
        and $S^{(1)}(t)$ is compact in the sense that for each $t>0$ the set $S^{(1)}(t)B$ is a relatively compact set in $X$ for every $t>0$ large enough and every bounded forward invariant set $B$ in $X$. Then $S(t)$ is asymptotically smooth.
    \end{prop}

    \begin{prop}[\cite{MR2438025}]\label{attractor_ex}
        Let $(X,S(t))$ be a dissipative dynamical system in a complete metric space $X$. Then $(X,S(t))$ possesses a compact global attractor if and only if $(X,S(t))$ is asymptotically smooth.
    \end{prop}
    
Given $\lambda>-\lambda_1$, we decompose $u=v+w$, where
    \begin{equation}\label{v_eq}
        \begin{cases}
            v_{tt} - \Delta v + \frac{1}{2}\|u_t\|^p u_t - \frac{1}{2}\|w_t\|^p w_t + \left(\|\nabla u\|^p + \frac{1}{2}\|u_t\|^p\right) v_t +\lambda v= 0, \\
          v|_{\partial \Omega} = 0, \\
          v(0) = u_0, v_t(0) = u_1,
        \end{cases}
    \end{equation}
    and
    \begin{equation}\label{w_eq}
        \begin{cases}
            w_{tt} - \Delta w + \left(\frac{1}{2}\|w_t\|^p + \|\nabla u\|^p + \frac{1}{2}\|u_t\|^p\right)w_t + f(u)-\lambda u +\lambda w= 0, \\
            w|_{\partial \Omega} = 0, \\
w(0) = w_t(0) = 0.
        \end{cases}
    \end{equation}
We point out that the parameter $\lambda$ is introduced in order to weaken the conditions on $f$ in Section \ref{sec_regularity-decay}. Lemmas \ref{v_decay} and \ref{w_regularity} below are established under Assumption \ref{f_assum1}.

    \begin{lemma}\label{v_decay}
        Let $B_0$ be a bounded subset of $\VV\times\HH$. Then any solution of \eqref{v_eq} starting from $B_0$ satisfies that
        \begin{equation*}
            \|\nabla v\|^2 + \|v_t\|^2 \rightarrow 0 \text{ as } t\to\infty,
        \end{equation*}
        In particular, the decay is uniform with respect to $(u_0, u_1)\in B_0$.
    \end{lemma}
\begin{proof}
  Multiplying equation \eqref{v_eq} by $v_t$ and integrating over $[0,t]\times\Omega$, we get
    \begin{equation*}
        \tilde{I}_v(t) - \tilde{I}_v(0) + 2\intt (g(v_t),v_t) = 0,
    \end{equation*}
    where 
    \begin{align*}
      \tilde{I}_v&=\|\nabla v\|^2 + \|v_t\|^2+\lambda \|v\|^2,\\
        g(v_t) &= \frac{1}{2}\|u_t\|^p u_t - \frac{1}{2}\|w_t\|^p w_t + \Big(\|\nabla u\|^p + \frac{1}{2}\|u_t\|^p\Big) v_t.
      \end{align*}
Note that $c_1I_v\leq \tilde{I}_v\leq c_2I_v$ for some $c_2>c_1>0$, as $\lambda>-\lambda_1$, where $I_v$ is defined by \eqref{I_u_definition}. By the monotonicity inequality (Lemma 2.2 in \cite{MR4064014}), it holds that
\begin{align*}
  (\|u_t\|^p u_t - \|w_t\|^p w_t, v_t)\geq C\|v_t\|^{p+2},
\end{align*}
and thereby
    \begin{equation}\label{g_inf}
        (g(v_t),v_t) \geq C\|v_t\|^{p+2} + \Big(\|\nabla u\|^p + \frac{1}{2}\|u_t\|^p\Big)\|v_t\|^2 \geq 0,
    \end{equation}
    which means that $\tilde{I}_v(t)$ is monotonically decreasing. This, together with Theorem \ref{dissipativity}, tells us that $\|\nabla u\|$, $\|u_t\|$, $\|\nabla v\|$ and $\|v_t\|$ are uniformly bounded for $t\in[0, \infty)$.

    To prove the lemma, we claim that, for fixed $T>0$ and any $M>0$, there exists a positive constant $K(M)$ such that
    \begin{equation*}
        \tilde{I}_v(T) - \tilde{I}_v(0) \leq -K(M),
    \end{equation*}
    whenever $(u_0,u_1)\in B_0$ and $\tilde{I}_v(T)\geq M$. Otherwise, there exist $M>0$ and a sequence $(v^n,v^n_t)$ associated with $(u_0^n,u_1^n)\in B_0$ such that
    \begin{align}\label{g_assum}
      \tilde{I}_{v^n}(T)\geq M~\mathrm{and}~\lim_{n\rightarrow \infty}\int_0^T(g(v^n),v^n_t)dt=0.
    \end{align}
    We infer from \eqref{g_inf} and \eqref{g_assum} that
    \begin{equation}\label{vn_t_converge}
        \|v^n_t\|\to 0 \trm{ in } L^{p+2}(0,T).
    \end{equation}
Besides, since $\{(v^n, v^n_t)\}$ is bounded in $L^{\infty}(0,T;\VV\times \HH)$, we can find a function $v$ such that, up to a subsequence
    \begin{gather}
        v^n \stackrel{*}{\rightharpoonup} v \trm{ in } L^{\infty}(0,T;\VV),\nonumber\\
        v^n \to v \trm{ a.e. in } [0,T]\times\Omega,\nonumber\\
        v^n \to v \trm{ in } C([0,T];L^2(\Omega)) \label{vn_converge}
    \end{gather}
    and also
    \begin{equation*}
        v^n_t\to 0, \trm{ a.e. in } [0,T]\times\Omega,
    \end{equation*}
where we have used Aubin-Lions Lemma in \eqref{vn_converge}. Since
    \begin{equation*}
        \begin{aligned}
            2\|g(v^n_t)\| 
            &\leq \left\| \|u^n_t\|^p u^n_t - \|u^n_t-v^n_t\|^p(u^n_t-v^n_t) \right\| + \big\| (2\|\nabla u^n\|^p+\|u^n_t\|^p) v^n_t \big\| \\
            &\leq \left| \|u^n_t\|^p-\|u^n_t-v^n_t\|^p \right|\|u^n_t\| + \|u^n_t-v^n_t\|^p\|v^n_t\| + (2\|\nabla u^n\|^p+\|u^n_t\|^p)\|v^n_t\| \\
            &\leq C_{B_0}\|v^n_t\|,
        \end{aligned}
    \end{equation*}
we also have $g(v^n_t)\to 0$ a.e. in $[0,T]\times\Omega$. Therefore, taking $n$ to $\infty$ in \eqref{v_eq}, one can see the limit function $v$ satisfies
    \begin{equation*}
        \begin{cases}
            -\Delta v+\lambda v = 0, & \mathrm{in}~\Omega\times[0, T],\\
            v|_{\partial\Omega} = 0, 
        \end{cases}
    \end{equation*}
which implies that $v = 0$ in $H_0^1(\Omega)$ for each $t\in[0, T]$. Furthermore, multiplying the equation of $v^n$ by $v^n$ and integrating over $[0,T]\times\Omega$, we know
    \begin{equation*}
        \int_0^T \|\nabla v^n\|^2dt = (v^n_t(0),v^n(0)) - (v^n_t(T),v^n(T)) + \int_0^T \|v^n_t\|^2 - (g(v^n_t), v^n)dt-\lambda\|v^n\|^2,
    \end{equation*}
which, together with \eqref{vn_t_converge} and \eqref{vn_converge}, infers that $\int_0^T \|\nabla v^n\|^2dt \to 0$ as $n$ tends to $\infty$. Therefore, it follows from the monotonicity of $\tilde{I}_v$ that
    \begin{equation*}
        M \leq \limsup_{n\to\infty} \tilde{I}_{v^n}(T) \leq \limsup_{n\rightarrow\infty}\frac{1}{T}\int_0^T \tilde{I}_{v^n}(t) dt=0,
    \end{equation*}
    which is a contraction.

    As the conclusion, the claim is true and, for any $\epsilon>0$, there exists $N_{\epsilon}\in\mathbb{N}^+$ such that
    \begin{equation*}
        0\leq \tilde{I}_v(t) \leq \epsilon \quad\textrm{for any } t\geq N_{\epsilon}T.
    \end{equation*}
By the equivalence of $\tilde{I}_v$ and $I_v$, we complete the proof.  
  \end{proof}
  
    \begin{lemma}\label{w_regularity}
       For each $t>0$, the solution $(w(t),w_t(t))$ of \eqref{w_eq} at $t$ is bounded in $H^{\beta+1}(\Omega)\times H^{\beta}(\Omega)$ for some $\beta>0$ small enough, where the bound is uniform with respect to $(u_0,u_1)\in B_0$. 
    \end{lemma}
\begin{proof}
Note that $(u, u_t)$ is globally bounded in $H^1_0(\Omega)\times L^2(\Omega)$, where the bound depends on $B_0$. For any $T>0$, by virtue of the strichartz estimate (Corollary 1.2 in \cite{MR2566711}), we know
    \begin{equation}\label{w_es_strichartz}
        \begin{split}
            \|u\|_{L^{7/2}(0,T;L^{14}(\Omega))} 
            \leq &C\Big(\|u(T/2)\|_{H^1(\Omega)} + \|u_t(T/2)\|_{\HH} + \|f(u)\|_{L^1(0,T;L^2(\Omega))} \\
            &\quad+ \big\|(\|\nabla u\|^p+\|u_t\|^p)u_t \big\|_{L^1(0,T;L^2(\Omega))}\Big) \\
            \leq &C(B_0) + C\|u\|_{L^3(0,T;L^6(\Omega))}^3+ C(T,B_0) \leq C(T,B_0).
        \end{split}
    \end{equation}
Denote $f_*(s)=f(s)-\lambda s$. It follows from the condition on $f$ that
    \begin{equation*}
            \|f_*(u)\|_{W^{1,6/5}(\Omega)}
            \leq C(1+\|u\|_{H^1(\Omega)}^3),
    \end{equation*}
    and 
    \begin{equation*}
            \|f_*(u)\|_{L^{7/6}(0, T; L^{14/3}(\Omega))} 
            \leq C(1+\|u\|_{L^{7/2}(0, T; L^{14}(\Omega))}^3),
    \end{equation*}
    which, together with the interpolation inequality $\|g\|_{H^{6/13}(\Omega)}\leq \|g\|_{W^{1,6/5}(\Omega)}^{6/13} \|g\|_{L^{14/3}(\Omega)}^{7/13}$, gives
    \begin{align*}
            \|f_*(u)\|_{L^{13/6}(0, T; H^{6/13}(\Omega))}\leq \|f_*(u)\|_{L^\infty(0, T; W^{1,6/5}(\Omega))}^{6/13} \|f_*(u)\|_{L^{7/6}(0, T; L^{14/3}(\Omega))}^{7/13}\leq C(T, B_0). 
     \end{align*}
Since $D(A^s)=H^{2s}(\Omega)$ for $s< 1/4$, we can denote $\tilde{w} = A^{\frac{3}{13}}w$ and consider the equation 
    \begin{equation*}
        \tilde{w}_{tt} - \Delta \tilde{w}+\lambda\tilde{w} + \Big(\frac{1}{2}\|w_t\|^p + \|\nabla u\|^p + \frac{1}{2}\|u_t\|^p\Big)\tilde{w}_t + A^{\frac{3}{13}}f_*(u) = 0
    \end{equation*}
    with $(\tilde{w}(0),\tilde{w}_t(0)) = 0$ and Dirichlet boundary. Then we know $(\tilde{w}(t),\tilde{w}_t(t))$ is bounded in $\VV\times\HH$ by energy estimate, which implies the conclusion with $\beta = \frac{6}{13}$.
  \end{proof}
  
Define $S^{(2)}(t)(u_0,u_1) = (v(t),v_t(t))$ and $S^{(1)}(t)(u_0,u_1) = (w(t),w_t(t))$. We know that $S^{(2)}(t)$ satisfies \eqref{S2_cond} by Lemma \ref{v_decay} and $S^{(1)}(t)$ is compact for each $t>0$ by Lemma \ref{w_regularity}. Therefore, by means of Proposition \ref{asym_smooth}, $S(t)$ is asymptotically smooth. Furthermore, since $S(t)$ is bounded dissipative by Theorem \ref{dissipativity}, we obtain directly from Proposition \ref{attractor_ex} the existence of the global attractor.

    \begin{thm}\label{attractor_existence}
        Under Assumption \ref{f_assum1}, the dynamical system generated by Problem \eqref{Problem} has a global attractor $\mathscr{A}$ in $\VV\times\HH$.
    \end{thm}

    \section{Regularity and Decay}\label{sec_regularity-decay}

    In this section, we establish some good properties, including the higher regularity of $\mathscr{A}$ and the polynomial decay rate of trajectories near the origin, which are useful to study the dimension of the global attractor. Besides Assumption \ref{f_assum1}, in this section we always assume the following condition:
    \begin{assum}\label{assum_2}
      $f(0)=0$, $f'(0)>-\lambda_1$.
    \end{assum}

We start with a technical lemma, which derives the boundedness from the uniform local information. 

    \begin{lemma}\label{F_estimate_lemma}
        Suppose that $F$, $\psi$ and $\varphi$ are nonnegative functions on $[0,\infty)$ and there exist constants $C_1, C_2$ such that
        \begin{itemize}
        \item [(i)] $ \|\psi\|_{L^1(t,t+1)} + \|\varphi\|_{L^1(t,t+1)} \leq C_1$ for any $t\geq 0$; 
        \item [(ii)]
        $\sup_{s\in[t,t+1]}\varphi(s)\leq C_2\inf_{s\in[t,t+1]}\varphi(s)$ for any $t\geq 0$;
      \item [(iii)] $F$ is absolutely continuous and satisfies $F'+ \varphi F \leq \psi\varphi$. 
      \end{itemize}
      Then there exists $M>0$, depending only on $F(0), C_1, C_2$, such that $F(t)\leq M$ for all $t\in [0, \infty)$.
    \end{lemma}
    \begin{proof}
      Since
    \begin{equation*}
        \left(e^{\intt \varphi(s)ds} F(t)\right)' \leq \psi(t)\varphi(t) e^{\intt \varphi(s) ds},
    \end{equation*}
    integrating over $[0,t]$, we can get
    \begin{align*}
            e^{\intt \varphi}F(t) - F(0) 
      \leq \int_0^t \psi(s)\varphi(s)e^{\int_0^s \varphi}ds\leq \sum_{k=1}^{\lceil t \rceil}\int_{k-1}^{k} \psi(s)\varphi(s)e^{\int_0^s \varphi}ds,
    \end{align*}
    where $\lceil t \rceil$ denotes the smallest integer larger than or equal to $t$. For each $k$,
      \begin{align*}
        \int_{k-1}^{k} \psi(s)\varphi(s)e^{\int_0^s \varphi}ds&\leq C_2\int_{k-1}^k \psi(s)ds\cdot \varphi(k) e^{\int_0^k \varphi}\leq C_1C_2 \int_{k-1}^k \varphi(k) e^{\int_0^k \varphi} ds \\
            &\leq C_1C_2^2e^{C_1} \int_{k-1}^k \varphi(s) e^{\int_0^s \varphi} ds,
      \end{align*}
      which implies that
      \begin{align*}
        e^{\intt \varphi}F(t) - F(0)\leq C\int_0^{\lceil t \rceil} \varphi(s) e^{\int_0^s \varphi} ds=C\big[e^{\int_0^{\lceil t\rceil}\varphi}-1\big].
      \end{align*}
Therefore,
\begin{align*}
  F(t)\leq F(0)e^{-\intt \varphi}+C\big[e^{\int_{t}^{\lceil t\rceil}\varphi}-e^{-\intt \varphi}\big]\leq M,
\end{align*}
where $M$ depends only on $F(0), C_1, C_2$. 
\end{proof}

    \begin{lemma}\label{w_regularity_new}
        Let $(u,u_t)$ start from a bounded set $B_0\subset \VV\times\HH$. Set $\lambda=f'(0)$ in equation \eqref{w_eq}. Then the solution $(w(t),w_t(t))$ keeps (globally in time) bounded in $H^{1+\beta}(\Omega)\times H^{\beta}(\Omega)$ with $\beta=\frac{2}{7}$.
    \end{lemma}
    \begin{proof}
     Similar to \eqref{w_es_strichartz}, for any $t$ and any $T>0$, we derive that
    \begin{equation*}
        \begin{aligned}
            \|u\|_{L^{\frac{7}{2}}(t,t+T;L^{14}(\Omega))} 
            &\leq C\Big( \|\nabla u(t+T/2)\| + \|u_t(t+T/2)\| +\|f(u)\|_{L^1(t,t+T;L^2(\Omega))}  \\
            &\qquad\qquad  + \big\|(\|\nabla u\|^p+\|u_t\|^p)u_t\big\|_{L^1(t,t+T;L^2(\Omega))} \Big) \\
            &\leq C(T, B_0).
        \end{aligned}
    \end{equation*}
Denote still $f_*(s)=f(s)-\lambda s$. Since $f_*(0)=f_*'(0)=0$ and $f_*\in C^2$, it holds that $|f_*'(s)|\leq C(|s|+|s|^2)$. Setting $\frac{1}{q} = \frac{1}{14}+\frac{1}{6}+\frac{1}{2}$, we calculate
    \begin{equation*}
        \begin{aligned}
            \|f_*(u)\|_{W^{1,q}}
            &\leq C\|f_*'(u)\nabla u\|_{L^q}\leq C\|(|u|+|u|^2)\nabla u\|_{L^q}\leq C(\|\nabla u\|^2 + \|u\|_{14}\|\nabla u\|^{2})\\&\leq C\left(1+\|u\|_{14}\right)\|\nabla u\|^2,
        \end{aligned}
      \end{equation*}
      and further, for $\beta=\frac{2}{7}$
      \begin{equation}\label{fu_Hs_estimate}
        \|f_*(u)\|_{D(A^{\frac{\beta}{2}})}\leq C\|f_*(u)\|_{H^\beta}\leq C\|f_*(u)\|_{W^{1,q}}\leq C\left(1+\|u\|_{14}\right)\|\nabla u\|^2.
      \end{equation}

      Denote $\tilde{w} = A^{\frac{\beta}{2}}w$, which satisfies
    \begin{equation}\label{w_tilde_eq}
      \begin{cases}
            \tilde{w}_{tt} - \Delta \tilde{w}+\lambda \tilde{w} + \left(\frac{1}{2}\|w_t\|^p + \|\nabla u\|^p + \frac{1}{2}\|u_t\|^p\right) \tilde{w}_t + A^{\frac{\beta}{2}}f_*(u) = 0, \\
        \tilde{w}(0) =\tilde{w}_t(0) = 0.
        \end{cases}
    \end{equation}
Multiplying \eqref{w_tilde_eq} by $\tilde{w}_t$ and integrating over $\Omega$, we deduce
    \begin{equation}\label{multi_w_tilde_t}
        \frac{1}{2}\frac{d}{dt}(\|\tilde{w}_t\|^2 + \|\nabla \tilde{w}\|^2+\lambda\|\tilde{w}\|^2) + \frac{1}{2}I_{u,p}\|\tilde{w}_t\|^2 + (A^{\frac{\beta}{2}}f_*(u),\tilde{w}_t) \leq 0,
      \end{equation}
where recall that $I_{u,p}=\|\nabla u\|^p + \|u_t\|^p$. Multiplying \eqref{w_tilde_eq} by $\tilde{w}$ and integrating over $\Omega$, we have
    \begin{equation}\label{multi_w_tilde}
        \begin{aligned}
            \frac{d}{dt}(\tilde{w}_t,\tilde{w}) - \|\tilde{w}_t\|^2 
            &+ \|\nabla \tilde{w}\|^2 +\lambda\|\tilde{w}\|^2+ (A^{\frac{\beta}{2}}f_*(u),\tilde{w}) \\
            &+ \Big(\frac{1}{2}\|w_t\|^p + \frac{1}{2}\|\nabla u\|^p + \frac{1}{2}I_{u,p}\Big)(\tilde{w}_t,\tilde{w}) = 0.
        \end{aligned}
    \end{equation}
    Then we multiply \eqref{multi_w_tilde} by $\epsilon I_u^{\frac{p}{2}}$ and sum it with \eqref{multi_w_tilde_t} to obtain
    \begin{equation}\label{w_tilde_estimate}
        \begin{aligned}
            &\frac{1}{2}\frac{d}{dt}(\|\tilde{w}_t\|^2 + \|\nabla \tilde{w}\|^2+\lambda\|\tilde{w}\|^2) 
            + \epsilon \frac{d}{dt}\left(I_u^{\frac{p}{2}}(\tilde{w}_t,\tilde{w})\right) +\Big(\frac{1}{2}-\epsilon\Big)I_u^{\frac{p}{2}}\|\tilde{w}_t\|^2\\
            & \quad+ \epsilon I_u^{\frac{p}{2}}(\|\nabla \tilde{w}\|^2+\lambda\|\tilde{w}\|^2) \leq \frac{p\epsilon}{2}\left|(\tilde{w}_t,\tilde{w})I_u^{\frac{p}{2}-1}\frac{d}{dt}I_u\right| + |(A^{\frac{\beta}{2}}f_*(u),\tilde{w}_t)|\\
            &\qquad + \epsilon I_u^{\frac{p}{2}}|(A^{\frac{\beta}{2}}f_*(u),\tilde{w})|  + C_p\epsilon I_u^{\frac{p}{2}}\Big(\frac{1}{2}\|w_t\|^p + \frac{1}{2}\|\nabla u\|^p + \frac{1}{2}I_{u,p}\Big)|(\tilde{w}_t,\tilde{w})|,
        \end{aligned}
      \end{equation}
      where the fact $I_u^{\frac{p}{2}}\approx I_{u,p}$ is used. For the first term on the right hand side of the above inequality, it holds 
    \begin{align*}            \frac{p\epsilon}{2}\left|(\tilde{w}_t,\tilde{w})I_u^{\frac{p}{2}-1}\frac{d}{dt}I_u\right| 
            &= \frac{p\epsilon}{2}\left| (\tilde{w}_t,\tilde{w})I_u^{\frac{p}{2}-1}\left[\effip\|u_t\|^2 + (f(u),u_t)\right] \right| \\
            &\leq C\epsilon I_u^{\frac{p}{2}}\|\tilde{w}_t\|\|\nabla\tilde{w}\|\\&\leq \frac{1}{8}I_u^{\frac{p}{2}}\|\tilde{w}_t\|^2 + C\epsilon^2 I_u^{\frac{p}{2}}\|\nabla\tilde{w}\|^2,
      \end{align*}
      where we have used the fact $|f(u)|\leq C(|u|+|u|^3)$. By \eqref{fu_Hs_estimate} and H\"older's inequality, we have
    \begin{equation*}
        \begin{aligned}
            |(A^{\frac{\beta}{2}}f_*(u),\tilde{w}_t)| 
            \leq C\|f_*(u)\|_{H^{\beta}} \|\tilde{w}_t\| 
            &\leq \frac{1}{8}\|\nabla u\|^2\|\tilde{w}_t\|^2 + C\|\nabla u\|^2\left(1 + \|u\|_{14}^2\right) \\
            &\leq \frac{1}{8}I_u^{\frac{p}{2}}\|\tilde{w}_t\|^2 + CI_u^{\frac{p}{2}}(1+\|u\|_{14}^2).
        \end{aligned}
    \end{equation*}
The last two terms can be estimated similarly since they all contain higher-power factors $I_u^{\frac{q}{2}}$ with $q\geq p$. As a result, for $\epsilon$ small enough we conclude
    \begin{align}
            &\frac{d}{dt}\Big(\frac{1}{2}(\|\tilde{w}_t\|^2 + \|\nabla \tilde{w}\|^2+\lambda\|\tilde{w}\|^2)+\epsilon I_u^{\frac{p}{2}}(\tilde{w}_t,\tilde{w})\Big) + \frac{2\epsilon}{3}I_u^{\frac{p}{2}}\Big(\frac{1}{2}(\|\tilde{w}_t\|^2 + \|\nabla \tilde{w}\|^2+\lambda\|\tilde{w}\|^2)\nonumber\\&\qquad+\epsilon I_u^{\frac{p}{2}}(\tilde{w}_t,\tilde{w})\Big)\leq CI_u^{\frac{p}{2}}(1+\|u\|_{14}^2).\label{r_es_Iu}
      \end{align}
      
Before applying Lemma \ref{F_estimate_lemma} to \eqref{r_es_Iu}, we still need to verify that
\begin{align}\label{r_cond_2}
  \sup_{s\in[t,t+1]}I_u(s)\leq C_2\inf_{s\in[t,t+1]}I_u(s)
\end{align}
for any $t\geq 0$. To do this, note that
\begin{align*}
  I_u'(t)+2I_{u, p}\|u_t\|^2+2(f(u), u_t)=0,
\end{align*}
which yields that
\begin{align*}
  -C(1+I_u)\leq -C\Big(I_u^{\frac{p}{2}}+\frac{|(f(u), u_t)|}{I_u}\Big)\leq \frac{I_u'}{I_u}&\leq \frac{2|(f(u), u_t)|}{I_u} \leq C(1+I_u).
\end{align*}
Thanks to the dissipation of $u$, we know $\Big|\frac{I_u'}{I_u}\Big|\leq C_{B_0}$, i.e., $|(\ln I_u)'|\leq C_{B_0}$. Therefore,
\begin{align*}
  e^{-C(t-s)}\leq \frac{I_u(t)}{I_u(s)}\leq e^{C(t-s)}\quad \mathrm{for~any~}0\leq s\leq t,
\end{align*}
which implies exactly \eqref{r_cond_2}.

Now, by applying Lemma \ref{F_estimate_lemma} to \eqref{r_es_Iu} and taking $\epsilon$ small enough such that
\begin{align*}
  \frac{1}{2}(\|\tilde{w}_t\|^2 + \|\nabla \tilde{w}\|^2+\lambda\|\tilde{w}\|^2)+\epsilon I_u^{\frac{p}{2}}(\tilde{w}_t,\tilde{w})\geq c(\|\tilde{w}_t\|^2 + \|\nabla \tilde{w}\|^2),
\end{align*}
we obtain the desired result. 
  \end{proof}
  
    From Lemma \eqref{v_decay} and Lemma \eqref{w_regularity_new}, we know there exists $R>0$ such that the set
    \begin{equation*}
       \big\{(u,v)\in H^{1+\beta}(\Omega)\times H^{\beta}(\Omega)\big| \|u\|_{H^{1+\beta}}^2 + \|v\|_{H^{\beta}}^2 \leq R^2\big\}
    \end{equation*}
    is attracting in $\VV\times\HH$. Therefore,
    \begin{thm}\label{attractor_reg}
        The global attractor $\mathscr{A}$ is bounded in $H^{\beta+1}(\Omega)\times H^{\beta}(\Omega)$ with $\beta=\frac{2}{7}$.
    \end{thm}

    \begin{coro}\label{Holder_t}
        The solutions in $\mathscr{A}$ are uniformly H\"older continuous in $t$, i.e. there exist $C>0$ and $\theta>0$, such that, for $(u, u_t)\in \mathscr{A}$
        \begin{equation*}
            \|(u(t_1),u_t(t_1))-(u(t_2),u_t(t_2))\|_{\VV\times\HH} \leq C|t_1-t_2|^{\theta}, \quad\forall\, t_1, t_2\in\mathbb{R},
        \end{equation*}
        where $C$ and $\theta$ are independent of $u$.
    \end{coro}
    \begin{proof}
By virtue of the interpolation inequality, we infer that
    \begin{equation*}
        \begin{aligned}
            \|u(t_1)-u(t_2)\|_{H^1_0}
            &\leq C\|u(t_1)-u(t_2)\|^{\theta} \|u(t_1)-u(t_2)\|^{1-\theta}_{H^{1+\beta}} \leq C_{\mathscr{A}}\Big\|\int_{t_1}^{t_2} u_t(s)ds\Big\|^{\theta} \\
            &\leq C_{\mathscr{A}}|t_2-t_1|^{\theta},
        \end{aligned}
    \end{equation*}
    where $\theta= \frac{\beta}{1+\beta}$, $\beta=\frac{2}{7}$, and similarly
    \begin{equation*}
        \begin{aligned}
            \|u_t(t_1)-u_t(t_2)\|
            &\leq C\|u_t(t_1)-u_t(t_2)\|_{H^{-1}}^{\theta} \|u_t(t_1)-u_t(t_2)\|_{H^{\beta}}^{1-\theta} \leq C_{\mathscr{A}} \Big\|\int_{t_1}^{t_2} u_{tt}(s)ds\Big\|^{\theta}_{H^{-1}} \\
            &\leq C_{\mathscr{A}}|t_2-t_1|^{\theta}.
        \end{aligned}
    \end{equation*}
  \end{proof}

    \begin{lemma}\label{u_decay_equiv}
        There exists $r_0>0$ such that any $(u,u_t)$ with initial data $(u_0,u_1)\in B(0,r_0)\subset \VV\times\HH$ converges to zero at a polynomial rate as follows:
        \begin{equation}\label{es_decay}
            C_1(t+k_1I_u(0)^{-\frac{p}{2}})^{-\frac{1}{p}} \leq \|(u,u_t)\|_{\VV\times\HH} \leq C_2(t+k_2I_u(0)^{-\frac{p}{2}})^{-\frac{1}{p}},
        \end{equation}
        where $I_u=\|\nabla u\|^2+\|u_t\|^2$ and $C_i,k_i>0$ are independent of $u$ and $t$.
    \end{lemma}

    \begin{proof}
      {\bf Part I.} Firstly we prove the left inequality by the continuity argument. Since this process will be used several times in the paper, we present here the details.

      Multiplying equation \eqref{Problem} by $u_t$ and integrating over $\Omega$, we have
    \begin{equation}\label{multiply_u_t}
        \frac{d}{dt} E(t) + I_{u, p} \|u_t\|^2 = 0,
    \end{equation}
    where $E(t)$ is defined in Lemma {\rm\ref{well-posed}}. Setting
    \begin{equation*}
        F_*(s)\triangleq F(s)-\frac{1}{2}f'(0)s^2 = \int_0^s\int_0^{\tau}\int_0^{\xi} f''(\zeta)d\zeta d\xi d\tau,
    \end{equation*}
we know from the assumption on $f$ that $|F_*(s)|\leq C(|s|^3 + |s|^4)$ for $s\in\mathbb{R}$ and thereby
    \begin{equation*}
        \left|\into F_*(u)dx \right| \leq C(\|\nabla u\|^3 + \|\nabla u\|^4).
    \end{equation*}
    Furthermore, since $f'(0)>-\lambda_1$ and
    \begin{equation*}
        E = \frac{1}{2}\|u_t\|^2 + \frac{1}{2}\|\nabla u\|^2 + \frac{f'(0)}{2}\|u\|^2 + \into F_*(u) dx,
      \end{equation*}
      there exists $\epsilon_1>0$ such that if $\|\nabla u\|\leq \epsilon_1$, we can find uniform constants $l_2>l_1>0$ satisfying
    \begin{equation}\label{E_equiv}
        l_1 I_u \leq E \leq l_2 I_u.
    \end{equation}

Let $r_0=\sqrt{\frac{l_1}{l_2}}\frac{\epsilon_1}{2}$ and $(u_0, u_1)\in B(0, r_0)\subset H^1_0(\Omega)\times L^2(\Omega)$, i.e. $I_u(0)<r_0^2=\frac{l_1\epsilon_1^2}{4l_2}$. We claim that $\|\nabla u(t)\|$ keeps smaller than $\epsilon_1$ for all $t\geq 0$. In fact, since $u\in C([0, \infty); H^1_0(\Omega))$ and $\|\nabla u_0\|<r_0<\frac{\epsilon_1}{2}$, there exists a maximal interval $[0, T)$ on which $\|\nabla u\|<\epsilon_1$ and \eqref{E_equiv} holds. 
Since $E(t)$ is monotonically decreasing in $t$, we have for $s\in [0, T)$ $E(s) \leq E(0) \leq l_2 I_u(0) = \frac{l_1\epsilon_1^2}{4}$ and further $I_u(s)\leq l_1^{-1}E(s)\leq \frac{\epsilon_1^2}{4}$. Therefore, 
    \begin{equation*}
        \|\nabla u(s)\| \leq I_u(s)^{\frac{1}{2}}\leq \frac{\epsilon_1}{2}\quad \mathrm{for}~s\in[0, T).
    \end{equation*}
    By the continuity argument, we conclude that $T=\infty$, i.e., the claim is true, and \eqref{E_equiv} holds for all $t\geq 0$.

    As a result, we infer from \eqref{multiply_u_t} and \eqref{E_equiv} that
    \begin{equation*}
            0\leq \frac{d}{dt} E(t) + I_{u,p}(t)I_u(t) \leq \frac{d}{dt} E(t) + C_pI_u(t)^{\frac{p}{2}+1} 
            \leq \frac{d}{dt} E(t) + C_pl_1^{-\frac{p}{2}-1} E^{\frac{p}{2}+1}(t),
    \end{equation*}
and thereby, for $t\geq 0$
    \begin{equation*}
       \|\nabla u(t)\|^2 + \|u_t(t)\|^2 \geq l_2^{-1}E(t) \geq l_2^{-1}\left(E(0)^{-\frac{p}{2}}+C t\right)^{-\frac{2}{p}}\geq l_2^{-1}\Big([l_1I_u(0)]^{-\frac{p}{2}}+C t\Big)^{-\frac{2}{p}}. 
    \end{equation*}
    This yields the first inequality by choosing $C_1, k_1$ properly.
    
    {\bf Part II.} Now we prove the right one. 
    Multiplying equation \eqref{Problem} by $u$ and integrating over $\Omega$, we have
    \begin{equation*}
        \frac{d}{dt}(u_t,u) - \|u_t\|^2 + \|\nabla u\|^2 + (f(u),u) + \effip(u_t,u) = 0.
    \end{equation*}
Note that
\begin{align*}
  f(s)s=f'(0)s^2+\int_0^s\int_0^\tau f''(\zeta)d\zeta d\tau\cdot s.
\end{align*}
Using the conditions on $f$ and the fact that $I_u$ keeps always small by the continuity argument in Part I, one can take $r_0$ small enough to obtain
    \begin{equation}\label{u_decay_equa2}
        \frac{d}{dt}(u_t,u) + C\|\nabla u\|^2 \leq 2\|u_t\|^2.
    \end{equation}
    Combining \eqref{multiply_u_t} and \eqref{E_equiv}, we also have
    \begin{equation}\label{u_decay_equa1}
        \frac{d}{dt}E(t) + C_pE(t)^{\frac{p}{2}}\|u_t\|^2 \leq 0.
    \end{equation}
    Now choose $\epsilon\in (0,\frac{C_p}{4})$. Multiplying \eqref{u_decay_equa2} by $\epsilon E(t)^{\frac{p}{2}}$ and summing it with \eqref{u_decay_equa1}, we get
    \begin{align*}
        \frac{d}{dt}E(t) + C_pE(t)^{\frac{p}{2}}\|u_t\|^2 + \epsilon E(t)^{\frac{p}{2}} \frac{d}{dt}(u_t,u) + \epsilon C E(t)^{\frac{p}{2}}\|\nabla u\|^2 &\leq 2\epsilon E(t)^{\frac{p}{2}}\|u_t\|^2,\\
            \left(1-\frac{\epsilon p}{2}E(t)^{\frac{p}{2}-1}(u_t,u)\right) \frac{d}{dt}E(t) 
            + \epsilon\frac{d}{dt}\left[E(t)^{\frac{p}{2}}(u_t,u)\right]+\frac{C_p}{2}&E(t)^{\frac{p}{2}}\|u_t\|^2 \\+ \epsilon C & E(t)^{\frac{p}{2}}\|\nabla u\|^2 \leq 0.
    \end{align*}
    Since $E(t)$ is decreasing, we deduce that for $r_0$ small enough
    \begin{align*}
      \left(1-\frac{\epsilon p}{2}E(t)^{\frac{p}{2}-1}(u_t,u)\right) \frac{d}{dt}E(t)\geq 2 \frac{d}{dt}E(t),
    \end{align*}
    and therefore
    \begin{equation*}
        2\frac{d}{dt}\left[E(t) + \frac{\epsilon}{2}E(t)^{\frac{p}{2}}(u_t,u)\right] + \epsilon C E(t)^{\frac{p}{2}}\effi \leq 0.
    \end{equation*}
    Noticing that the quantities $I_u$, $E(t)$ and $E(t) + \frac{\epsilon}{2}E(t)^{\frac{p}{2}}(u_t,u)$ are equivalent, we end up with
    \begin{equation*}
        \frac{d}{dt}\left[E(t) + \frac{\epsilon}{2}E(t)^{\frac{p}{2}}(u_t,u)\right] + C\left[E(t) + \frac{\epsilon}{2}E(t)^{\frac{p}{2}}(u_t,u)\right]^{\frac{p}{2}+1} \leq 0.
    \end{equation*}
    Therefore, by choosing properly $C_2$ and $k_2$, the desired inequality follows.
  \end{proof}
  \begin{remark}\label{regul-decay_remark_decay}
    Lemma \ref{u_decay_equiv} can be partially strengthened as follows:

    For each $r>0$, there exist $C_1, k_1$, depending on $r$, such that any $(u,u_t)$ with initial data $(u_0,u_1)\in B(0,r)\subset \VV\times\HH$ satisfies 
    \begin{equation*}
      \|(u,u_t)\|_{\VV\times\HH}\geq C_1(t+k_1I_u(0)^{-\frac{p}{2}})^{-\frac{1}{p}}.
    \end{equation*}

To do this, it is sufficient to choose $C_1$ smaller such that $C_1\leq k_1^{\frac{1}{p}}$. In fact, if $\|(u_0, u_1)\|_{\VV\times\HH}>r_0$ and $\|(u, u_t)\|_{\VV\times\HH}$ reaches $r_0$ for the first time at $t_0$, we know from Lemma \ref{u_decay_equiv} that for this smaller $C_1$ and all $t\geq t_0$
\begin{align*}
  \|(u, u_t)\|_{\VV\times\HH}&\geq C_1(t-t_0+k_1r_0^{-p})^{-\frac{1}{p}}\geq C_1(t+k_1r_0^{-p}r^pI_u(0)^{-\frac{p}{2}})^{-\frac{1}{p}}.
\end{align*}
On the other hand, since $\|(u, u_t)\|_{\VV\times\HH}\geq r_0$ for $t\in[0, t_0]$, it also holds
\begin{align*}
  \|(u, u_t)\|_{\VV\times\HH}\geq r_0\geq C_1(k_1r_0^{-p})^{-\frac{1}{p}}\geq C_1(t+k_1r_0^{-p}r^pI_u(0)^{-\frac{p}{2}})^{-\frac{1}{p}},
\end{align*}
where $C_1\leq k_1^{\frac{1}{p}}$ is used. This gives the desired estimate by replacing $k_1$ with $k_1r_0^{-p}r^p$. 
  \end{remark}

    \section{Dimension estimates}\label{sec_dimensions}
        
    \begin{Def}
        Let $X$ be a metric space and $K\subset X$. Let $N(K,\epsilon)$ denote the minimum number of closed balls of radius $\epsilon$ with centres in $K$ required to cover $K$. The fractal dimension (or upper box-counting dimension) of $K$ is defined by
        $$
        d_B(K) = \limsup_{\epsilon\to 0} \frac{\ln N(K,\epsilon)}{-\ln \epsilon}.
        $$
    \end{Def}
    Usually it is simpler to calculate the fractal dimension by taking the (superior) limit through a discrete sequence $\{\epsilon_k\}$, rather than a continuous one, as in the following lemma.

    \begin{lemma}[\cite{MR2767108}]\label{box_dim}
        If $\{\epsilon_k\}$ is a decreasing sequence tending to zero with $\epsilon_{k+1} \geq \alpha \epsilon_k$ for some $\alpha\in (0,1)$, then
        $$
        d_B(K) = \limsup_{k\to \infty} \frac{\ln N(K,\epsilon_k)}{-\ln \epsilon_k}.
        $$
    \end{lemma}

    Consider the following degenerate evolutionary problem
    \begin{equation}\label{abstract_equa}
        u_t + Au = 0
    \end{equation}
on a Hilbert space $X$, where $A$ is a nonlinear unbounded operator and is degenerate at the origin. Suppose that this abstract problem possesses a global attractor $\mathscr{A}$. We would like to figure out the fractal dimension of $\mathscr{A}$. However, it is not able to apply directly the method of \cite{MR2767108} in this setting because the derivative of $S(t)$ may not be a compact perturbation of a contractive operator near the origin where the degeneration occurs. By the similar reason, the quasi-stability method does not work either, since the quasi-stability inequality is violated near the origin. 

    Observing that $A$ is non-degenerate outside of a neighborhood of the origin, we embrace tentatively the belief that the properties of $\mathscr{A}$ away from the origin are good and would like to deal with the degenerate region primarily.

    \begin{thm}\label{main}
        Let $\mathscr{A}$ be the global attractor of the dynamic system $S(t)$ on a Hilbert space $(X,\Vert\cdot\Vert)$. Assume that $0\in \mathscr{A}$ and the fractal dimension of $\mathscr{A}\backslash B(0,\epsilon)$ is finite for each $\epsilon > 0$. Let $\{\epsilon_m\}_{m=0}^{\infty}$ be a decreasing sequence and $\{t_m\}_{m=0}^{\infty}$ an increasing sequence with $t_0=0$, such that
        \begin{gather}
          \epsilon_m \rightarrow 0,~~\epsilon_{m+1} \geq \alpha \epsilon_m \quad\textrm{for some }\alpha\in (0,1),\label{cond_epsilon}\\
            \sup_{u\in \mathscr{A}\cap \overline{B(0,\epsilon_0)}}\|S(t)u\|\leq \epsilon_m \quad \textrm{for }t\geq t_m\label{cond_decay}
        \end{gather}
        and
        \begin{equation}\label{cond_limit}
            d_0 \stackrel{\Delta}{=} \limsup_{m\to\infty} \frac{\ln t_m}{-\ln \epsilon_m} <\infty.
        \end{equation}
        Suppose further that $S(t)u_0$ is uniformly H\"older continuous on $(0,\infty)\times[\mathscr{A}\cap B(0,\epsilon_0)]$, i.e., there exist $\theta\in(0,1]$ and $L>0$ such that for $t_1, t_2>0$ and $u_1, u_2\in \mathscr{A}\cap B(0,\epsilon_0)$
        \begin{equation}\label{Holder_cond}
            \|S(t_1)u_1-S(t_2)u_2\|\leq L(|t_1-t_2|+\|u_1-u_2\|)^{\theta}.
        \end{equation}
        Then $d_B(\mathscr{A}) < \infty$.
    \end{thm}
    \begin{proof}
      Denote $\mathscr{A}_{\epsilon_0}=\mathscr{A}\cap B(0, \epsilon_0)$ and $\mathscr{A}_{\epsilon_m}^c=\mathscr{A}\cap [B(0, \epsilon_0)\setminus \overline{B(0, \epsilon_m)}]$. Without loss of generality, we suppose that $\epsilon_1<\epsilon_0$. It is easy to verify that $\cup_{s=0}^tS(s) \mathscr{A}_{\epsilon_1}^c\supset \mathscr{A}_{\epsilon_m}^c$ for $t\geq t_m$ by \eqref{cond_decay}. For each $m$, we have
      \begin{align*}
        N(\mathscr{A}_{\epsilon_1}^c\times [0, t_m], \sqrt{2}\epsilon_m)\leq N(\mathscr{A}_{\epsilon_1}^c, \epsilon_m)\times \lceil\epsilon_m^{-1}t_m\rceil,
      \end{align*}
      where $\lceil s\rceil$ denotes the smallest integer larger than or equal to $s$. This means $\mathscr{A}_{\epsilon_1}^c\times [0, t_m]$ has a covering of no more than $N(\mathscr{A}_{\epsilon_1}^c, \epsilon_m)\times \lceil\epsilon_m^{-1}t_m\rceil$ balls of radius $\sqrt{2}\epsilon_m$. The image of this covering under $S(\cdot)$ provides a covering of $\cup_{s=0}^{t_m}S(s) \mathscr{A}_{\epsilon_1}^c$ by sets of diameter no larger than $L(2\sqrt{2}\epsilon_m)^\theta$. Each of these sets is certainly contained in a closed ball of radius $2L(2\sqrt{2}\epsilon_m)^\theta$, i.e., $$N(\cup_{s=0}^{t_m}S(s) \mathscr{A}_{\epsilon_1}^c, 2L(2\sqrt{2}\epsilon_m)^\theta)\leq N(\mathscr{A}_{\epsilon_1}^c, \epsilon_m)\times \lceil\epsilon_m^{-1}t_m\rceil.$$ Note that $2L(2\sqrt{2}\epsilon_m)^\theta\geq \epsilon_m$ for $\epsilon_m$ small enough. In this situation, it follows that
      \begin{align*}
        N(\mathscr{A}_{\epsilon_0}, 2L(2\sqrt{2}\epsilon_m)^\theta)&\leq N(\mathscr{A}_{\epsilon_m}^c, 2L(2\sqrt{2}\epsilon_m)^\theta)+1\leq N(\cup_{s=0}^{t_m}S(s) \mathscr{A}_{\epsilon_1}^c, 2L(2\sqrt{2}\epsilon_m)^\theta)+1\\&\leq N(\mathscr{A}_{\epsilon_1}^c, \epsilon_m)\times \lceil\epsilon_m^{-1}t_m\rceil+1,
      \end{align*}
and thereby
\begin{align*}
  d_B(\mathscr{A})&\leq d_B(\mathscr{A}\setminus \mathscr{A}_{\epsilon_0})+\limsup_{m\to \infty} \frac{\ln N(\mathscr{A}_{\epsilon_0}, 2L(2\sqrt{2}\epsilon_m)^\theta)}{-\ln [2L(2\sqrt{2}\epsilon_m)^\theta]}\\&\leq d_B(\mathscr{A}\setminus \mathscr{A}_{\epsilon_0})+\limsup_{m\to \infty} \frac{\ln [N(\mathscr{A}_{\epsilon_1}^c, \epsilon_m)\times \lceil\epsilon_m^{-1}t_m\rceil+1]}{-\theta\ln\epsilon_m-\ln (2^{1+3\theta/2}L)}\\&=d_B(\mathscr{A}\setminus \mathscr{A}_{\epsilon_0})+\limsup_{m\to \infty} \frac{\ln N(\mathscr{A}_{\epsilon_1}^c, \epsilon_m)+\ln t_m-\ln \epsilon_m}{-\theta\ln\epsilon_m}\\&=d_B(\mathscr{A}\setminus \mathscr{A}_{\epsilon_0})+\theta^{-1}d_B(\mathscr{A}_{\epsilon_1}^c)+\theta^{-1}d_0+\theta^{-1}\\&<\infty.
\end{align*}
This completes the proof.
    \end{proof}  

    Intuitively, conditions \eqref{cond_epsilon}-\eqref{cond_limit} can be replaced by the following assumption of Theorem \ref{main_2} in the setting of problem \eqref{abstract_equa}.

    \begin{thm}\label{main_2}
      Let $\mathscr{A}$ be the global attractor of the dynamic system $S(t)$ generated by equation \eqref{abstract_equa} on Hilbert space $(X,\Vert\cdot\Vert)$. Assume that $0\in \mathscr{A}$ and the fractal dimension of $\mathscr{A}\backslash B(0,\epsilon)$ is finite for each $\epsilon > 0$. Suppose that there exist $\epsilon_1> 0$, $C>0$, $\alpha > 0$, $l_2\geq 1_1>0$ and a functional $E(u)$ defined on $\mathscr{A}\cap B(0, \epsilon_1)$ such that
      \begin{align}\label{dimensions_es_equiv}
        l_1\|u\|^2\leq E(u)\leq l_2\|u\|^2\quad \textrm{for }u\in \mathscr{A}\cap B(0, \epsilon_1),
      \end{align}
       and, for $u(t)=S(t)u_0$
      \begin{equation}\label{decrease_cond}
        \frac{d}{dt}E(u(t)) + CE(u(t))^{1+\alpha} \leq 0,\quad \textrm{if }u(t)\in \mathscr{A}\cap B(0, \epsilon_1). 
      \end{equation}
Suppose further that $S(t)u_0$ is uniformly $\theta$-H\"older continuous on $(0,\infty)\times[\mathscr{A}\cap B(0,\epsilon_1)]$. Then $d_B(\mathscr{A}) < \infty$.
    \end{thm}
    \begin{proof}
      Set $\epsilon_0=l_2^{-1/2}l_1^{1/2}\epsilon_1$. It is easy to verify that $u(t)$ stays in $\mathscr{A}\cap B(0, \epsilon_1)$ for all $t\geq 0$ if $u_0\in \mathscr{A}\cap B(0,\epsilon_0)$. Therefore, it is sufficient to show there exist a decreasing sequence $\{\epsilon_m\}_{m=0}^{\infty}$ with $\epsilon_0$ given above and an increasing sequence $\{t_m\}_{m=0}^{\infty}$ with $t_0=0$ satisfying \eqref{cond_epsilon}-\eqref{cond_limit}. Multiplying \eqref{decrease_cond} by $E(u)^{-1-\alpha}$ and integrating from 0 to $t$, we obtain
    \begin{align*}
      E(u(t))^{-\alpha} &\geq\alpha Ct + E(u(0))^{-\alpha},\\
      E(u(t)) &\leq (\alpha Ct + E(u(0))^{-\alpha})^{-\frac{1}{\alpha}},
    \end{align*}
    which implies that
    \begin{align*}
      \|u(t)\|^2 &\leq (\alpha l_1^\alpha C t + l_1^\alpha l_2^{-\alpha}\|u(0)\|^{-2\alpha})^{-\frac{1}{\alpha}}. 
    \end{align*}
For $m\geq 1$, set $\epsilon_m = 2^{-m}\epsilon_0$ and $t_m =\frac{2^{2m\alpha}-l_1^\alpha l_2^{-\alpha}}{\alpha l_1^\alpha C\epsilon_0^{2\alpha}}$. One can easily verify that \eqref{cond_epsilon} and \eqref{cond_decay} hold. Besides, we infer that
    \begin{equation*}
        \limsup_{m\to\infty} \frac{\ln t_m}{-\ln \epsilon_m} = \limsup_{m\to\infty}\frac{\ln(2^{2m\alpha}-l_1^\alpha l_2^{-\alpha}) - \ln(\alpha l_1^\alpha C\epsilon_0^{2\alpha})}{m\ln 2 - \ln \epsilon_0} = 2\alpha,
    \end{equation*}
    which implies \eqref{cond_limit}. 
  \end{proof}

    \section{Finite dimensionality}\label{sec_application}

    In this section, we use Theorem \ref{main_2} to prove the finite dimensionality of the global attractor for problem \eqref{Problem}. 

    The formal linearization of \eqref{Problem} is given by
    \begin{equation}\label{diff_eq} 
        \begin{cases}
            U_{tt} - \Delta U + \effip U_t + f'(u)U \\
            \qquad\qquad\qquad  + p\left[\|\nabla u\|^{p-2}(\nabla u, \nabla U)+\|u_t\|^{p-2}(u_t,U_t)\right]u_t = 0, \\
            U(0)=\xi, U_t(0)=\zeta. 
        \end{cases}
    \end{equation}
    
    \begin{lemma}
        Suppose Assumption \ref{f_assum1} holds. For $t>0$, $S(t)$ is Fr\'echet differentiable on $\VV\times\HH$. Its derivative at $\omega_0=(u_0,u_1)$ is the linear operator
        \begin{equation*}
            L(t; \omega_0): (\xi,\zeta)\mapsto (U(t),U_t(t)),
        \end{equation*}
        where $U$ is the solution of \eqref{diff_eq}. Furthermore, for each $t>0$, $L(t; \omega_0)$ is continuous in $\omega_0$.
    \end{lemma}
    \begin{proof}
      Similarly to the proof of Theorem \ref{well-posed}, one may show the existence and uniqueness of \eqref{diff_eq} in $\VV\times\HH$ by the monotone operator theory. We omit the details and prove only $DS(t)\omega_0= L(t; \omega_0)$.

      Denote $X=\VV\times\HH$. Let $\omega_0=(u_0,u_1)\in X$ and $\tilde{\omega}_0=\omega_0 + (\xi,\zeta)\in X$. For $t\geq 0$, denote $\omega(t)=S(t)\omega_0=(u(t),u_t(t))$ and $\tilde{\omega}(t)=S(t)\tilde{\omega}_0=(\tilde{u}(t),\tilde{u}_t(t))$, both of which are uniformly bounded in $X$ due to the dissipation, with the bound depending on $\|\omega_0\|$ and $\|\tilde{\omega}_0\|$.
      
      We first show $S(t)$ is locally Lipschitz in $X$ for each $t\geq 0$. Indeed, the difference $\psi=\tilde{u}-u$ satisfies
    \begin{align}\label{psi_eq}
      \begin{cases}
        \psi_{tt} - \Delta\psi + (\|\nabla \tilde{u}\|^p+\|\tilde{u}_t\|^p)\tilde{u}_t - (\|\nabla u\|^p+\|u_t\|^p) u_t + f(\tilde{u}) - f(u) = 0,  \\
        \psi(0)=\xi, \psi_t(0)=\zeta.
      \end{cases}
    \end{align}
    Multiplying \eqref{psi_eq} by $\psi_t$ and integrating over $\Omega$, we have
    \begin{equation*}
        \begin{aligned}
            &\quad\frac{1}{2}\frac{d}{dt}\|\psi_t\|^2 + \frac{1}{2}\frac{d}{dt}\|\nabla \psi\|^2 + (\|\tilde{u}_t\|^p\tilde{u}_t-\|u_t\|^pu_t,\psi_t) + \|\nabla \tilde{u}\|^p\|\psi_t\|^2 \\
            &= (\|\nabla u\|^p-\|\nabla\tilde{u}\|^p)(u_t,\psi_t) + (f(u)-f(\tilde{u}),\psi_t) \\
            &= -\int_0^1 p\|\nabla u_{\theta}\|^{p-2}(\nabla u_{\theta},\nabla \psi)d\theta\cdot (u_t,\psi_t)- \into\int_0^1 f'(u_\theta)\psi d\theta\cdot \psi_t dx \\
            &\leq C(\|\nabla \psi\|^2 + \|\psi_t\|^2),
        \end{aligned}
    \end{equation*}
where $u_{\theta} = \theta u + (1-\theta)\tilde{u}$. Noticing that $(\|\tilde{u}_t\|^p\tilde{u}_t-\|u_t\|^pu_t,\psi_t)\geq 0$, by means of the Gronwall lemma we have
    \begin{equation}\label{a_es_Lip}
        \|\tilde{\omega}(t)-\omega(t)\|_X^2 \leq e^{2Ct}\|(\xi,\zeta)\|_X^2.
    \end{equation}

To continue, we introduce the operators from $C([0,T];\VV)\cap C^1([0,T];\HH)$ to $C([0, T]; L^2(\Omega))$ by
    \begin{equation*}
        \mathcal{G}(u)= \|\nabla u\|^p u_t,\quad \mathcal{H}(u)= \|u_t\|^p u_t.
    \end{equation*}
    Their Fr\'echet derivatives are given by
    \begin{align*}
        \mathcal{G}'(u)v &= \|\nabla u\|^p v_t + p\|\nabla u\|^{p-2}(\nabla u,\nabla v)u_t,\\
        \mathcal{H}'(u)v &= \|u_t\|^p v_t + p\|u_t\|^{p-2}(u_t,v_t)u_t.
    \end{align*}
    For each $t$, we set
    \begin{align*}
        \mathfrak{G} &\stackrel{\Delta}{=} \mathcal{G}(\tilde{u})-\mathcal{G}(u)-\mathcal{G}'(u)(\tilde{u}-u) = \mathfrak{G}_1+\mathfrak{G}_2,\\
        \mathfrak{H} &\stackrel{\Delta}{=} \mathcal{H}(\tilde{u})-\mathcal{H}(u)-\mathcal{H}'(u)(\tilde{u}-u),
    \end{align*}
    where
    \begin{align*}
      \mathfrak{G}_1& = (\|\nabla\tilde{u}\|^p-\|\nabla u\|^p)u_t - p\|\nabla u\|^{p-2}(\nabla u,\nabla \tilde{u} - \nabla u)u_t,\\
      \mathfrak{G}_2&=(\|\nabla\tilde{u}\|^p-\|\nabla u\|^p)(\tilde{u}_t-u_t). 
    \end{align*}
    For any $\epsilon>0$, since $S(t)$ is Lipschitz, we have, if $(\xi,\zeta)$ is close to zero
    \begin{equation*}
        \begin{aligned}
            \|\mathfrak{G}_1\| 
            &= \Big\| p\int_0^1 \|\nabla u_{\theta}\|^{p-2}(\nabla u_{\theta},\nabla \tilde{u}-\nabla u)d\theta\cdot u_t - p(\|\nabla u\|^{p-2}\nabla u,\nabla \tilde{u}-\nabla u)u_t \Big\| \\
            &\leq p\int_0^1 \big\|\|\nabla u_{\theta}\|^{p-2}\nabla u_{\theta}-\|\nabla u\|^{p-2}\nabla u \big\|d\theta\cdot\|\nabla \tilde{u}-\nabla u\|\|u_t\| \\
            &\leq \frac{\epsilon}{2}\|\nabla \tilde{u} - \nabla u\|
        \end{aligned}
    \end{equation*}
and
    \begin{equation*}
        \|\mathfrak{G}_2\| \leq \frac{\epsilon}{2}\|\tilde{u}_t-u_t\|.
    \end{equation*}
Therefore, we infer that
    \begin{equation}\label{g_overline_estimate}
        \|\mathfrak{G}\| \leq \epsilon\|\tilde{\omega}-\omega\|_X,
    \end{equation}
    and following a similar argument, 
    \begin{equation}\label{h_overline_estimate}
        \|\mathfrak{H}\| \leq \epsilon\|\tilde{\omega}-\omega\|_X.
    \end{equation}

    Now, denote $\Psi=\tilde{u}-u-U$ with $U$ the solution of \eqref{diff_eq}. By calculation we know that $\Psi$ satisfies
    \begin{equation*}
        \begin{aligned}
            \Psi_{tt} - \Delta\Psi 
            &+ \mathcal{G}(\tilde{u})- \mathcal{G}(u)-\mathcal{G}'(u)U+\mathcal{H}(\tilde{u})-\mathcal{H}(u) - \mathcal{H}'(u)U \\
            &+ f(\tilde{u}) - f(u) -f'(u)U = 0.
        \end{aligned}
    \end{equation*}
    Denoting $\mathfrak{F}=f(\tilde{u})-f(u)-f'(u)(\tilde{u}-u)$, we can rewrite the above equation into
    \begin{equation}\label{theta_eq}
        \begin{aligned}
            \Psi_{tt} - \Delta\Psi + \mathcal{G}'(u)\Psi+\mathcal{H}'(u)\Psi + f'(u)\Psi+\mathfrak{G}+\mathfrak{H}+\mathfrak{F} =0.
        \end{aligned}
    \end{equation}
    Noticing that
    \begin{equation*}
        \mathfrak{F}(t) 
            = \int_0^1 [f'(\theta \tilde{u}+(1-\theta)u)-f'(u)] (\tilde{u}-u) d\theta = \int_0^1\int_0^1 f''(u+\tau\theta(\tilde{u}-u))\theta\psi^2d\tau d\theta,
    \end{equation*}
    it holds
    \begin{equation*}
        |(\mathfrak{F}(t),\Psi_t)| \leq C\|(1+u+\tilde{u})\|_6\|\psi^2\|_3\|\Psi_t\|\leq C\|\nabla \psi\|^2\|\Psi_t\|.
    \end{equation*}
    Multiplying \eqref{theta_eq} by $\Psi_t$ and integrating over $\Omega$, we obtain
    \begin{equation*}
        \begin{aligned}
            \frac{1}{2}\frac{d}{dt}(\|\Psi_t\|^2+\|\nabla \Psi\|^2) 
            &= -([\mathcal{G}'(u)+\mathcal{H}'(u) + f'(u)]\Psi,\Psi_t) - (\mathfrak{G}+\mathfrak{H}+\mathfrak{F},\Psi_t) \\
            &\leq C(\|\Psi_t\|^2 + \|\nabla \Psi\|^2) + (\|\mathfrak{G}\|+\|\mathfrak{H}\|)\|\Psi_t\| + C\|\nabla \psi\|^2\|\Psi_t\|.
        \end{aligned}
    \end{equation*}
    Let $(\xi,\zeta)$ be close enough to zero. Then, by \eqref{a_es_Lip}--\eqref{h_overline_estimate}, we have
    \begin{equation*}
        \begin{aligned}
            \frac{1}{2}\frac{d}{dt}(\|\Psi_t\|^2+\|\nabla \Psi\|^2) 
            &\leq C(\|\Psi_t\|^2 + \|\nabla \Psi\|^2) + (2\epsilon\|\tilde{\omega}-\omega\|_X+C\|\nabla \psi\|^2)\|\Psi_t\| \\
            &\leq C(\|\Psi_t\|^2 + \|\nabla \Psi\|^2) + 3\epsilon\|\tilde{\omega}-\omega\|_X\|\Psi_t\| \\
            &\leq (C+1)(\|\Psi_t\|^2 + \|\nabla \Psi\|^2) + \frac{9}{4}\epsilon^2\|\tilde{\omega}-\omega\|_X^2.
        \end{aligned}
    \end{equation*}
    Therefore, it follows from the Gronwall lemma, as well as \eqref{a_es_Lip}, that
    \begin{equation*}\label{theta_estimate}
        \|\Psi\|_X^2 \leq \frac{9}{4}\epsilon^2 e^{C t}\|(\xi,\zeta)\|_X^2,
    \end{equation*}
    since $\Psi(0) = 0$. In other words,
    \begin{equation*}\label{theta_limit}
        \frac{\|\tilde{\omega}(t)-\omega(t)-(U(t),U_t(t))\|_X^2}{\|(\xi,\zeta)\|_X^2} \leq \frac{9}{4}\epsilon^2 e^{Ct} \quad\textrm{for } (\xi,\zeta) \textrm{ close enough to }0 \textrm{ in } X.
    \end{equation*}
    Since $\epsilon$ could be arbitrarily small, we obtain the differentiability of $S(t)$. 
  \end{proof}
  
    \begin{lemma}\label{Lipschitz_u}
        Suppose that $f$ satisfies Assumption \ref{f_assum1} and \eqref{f_cond3}. Then there exists $r_0>0$ such that $S(t)\omega_0$ is uniformly (with respect to $t$) Lipschitz continuous on $B(0,r_0)\subset H^1_0(\Omega)\times L^2(\Omega)$.
    \end{lemma}
    \begin{proof}
      Let $\|(\xi,\zeta)\|_X\leq 1$ and denote $f_*(s)=f(s)-f'(0)s$. We multiply equation \eqref{diff_eq} by $U_t$ and integrate over $\Omega$ to get
      \begin{align}\label{diff_multi_U_t}
        \begin{split}
           &\quad\frac{1}{2}\frac{d}{dt}[\|U_t\|^2
            + \|\nabla U\|^2+f'(0)\|U\|^2]
            + (\|\nabla u\|^p + \|u_t\|^p)\|U_t\|^2\\&
            + p\|\nabla u\|^{p-2}(\nabla u,\nabla U)(u_t,U_t)
            + p\|u_t\|^{p-2}|(u_t,U_t)|^2
            + (f_*'(u)U,U_t)
            = 0.
        \end{split}
      \end{align}
    And similarly, by the test function $U$, it holds
    \begin{align}\label{diff_multi_U}
      \begin{split}
        &\quad\frac{d}{dt}(U_t,U)
            - \|U_t\|^2
            + \|\nabla U\|^2+f'(0)\|U\|^2
            + (\|\nabla u\|^p + \|u_t\|^p)(U_t,U)\\&
            + p\|\nabla u\|^{p-2}(\nabla u,\nabla U)(u_t,U)
            + p\|u_t\|^{p-2}(u_t,U_t)(u_t,U)
            + (f_*'(u)U,U)
            = 0.
      \end{split}
    \end{align}
    Now, taking $\epsilon>0$, we multiply \eqref{diff_multi_U} with $\epsilon I_u^{\frac{p}{2}}(t)$ and sum it with \eqref{diff_multi_U_t} to obtain
    \begin{align}\label{diff_multi_E2p_1}
      \begin{split}
      &\quad\frac{1}{2}\frac{d}{dt}(\|U_t\|^2 + \|\nabla U\|^2+f'(0)\|U\|^2)
            + \epsilon I_u^{\frac{p}{2}} \frac{d}{dt}(U_t,U)+ p\|u_t\|^{p-2}|(u_t,U_t)|^2
            \\&+ I_{u,p}\|U_t\|^2
            + \epsilon I_u^{\frac{p}{2}}(\|\nabla U\|^2+f'(0)\|U\|^2)
            = -p\|\nabla u\|^{p-2}(\nabla u,\nabla U)(u_t,U_t)\\&
            - (f_*'(u)U,U_t)
            + \epsilon I_u^{\frac{p}{2}}\|U_t\|^2
            - \epsilon I_u^{\frac{p}{2}}I_{u,p}(U_t,U)
            - \epsilon pI_u^{\frac{p}{2}}\|\nabla u\|^{p-2}(\nabla u,\nabla U)(u_t,U)\\&
            - \epsilon pI_u^{\frac{p}{2}}\|u_t\|^{p-2}(u_t,U_t)(u_t,U)
            - \epsilon I_u^{\frac{p}{2}}(f_*'(u)U,U),
      \end{split}
    \end{align}
    where $I_u, I_{u,p}$ have been given in \eqref{I_u_definition}. Note that
    \begin{equation*}
        \begin{aligned}
            \epsilon I_u^{\frac{p}{2}}\frac{d}{dt}(U_t,U) 
            &= \epsilon \frac{d}{dt}[I_u^{\frac{p}{2}}(U_t,U)] - \frac{\epsilon p}{2}(U_t,U)I_u^{\frac{p}{2}-1}\frac{d}{dt}(\|\nabla u\|^2 + \|u_t\|^2) \\
            &= \epsilon \frac{d}{dt}[I_u^{\frac{p}{2}}(U_t,U)] - \frac{\epsilon p}{2}(U_t,U)I_u^{\frac{p}{2}-1}\left[-I_{u,p}\|u_t\|^2 - (f(u),u_t)\right] \\
            &\geq \epsilon \frac{d}{dt}[I_u^{\frac{p}{2}}(U_t,U)] - \epsilon C(\|U_t\|^2+\|\nabla U\|^2)I_u^{\frac{p}{2}-1}\left[I_{u,p}\|u_t\|^2 + |(f(u),u_t)|\right],
        \end{aligned}
      \end{equation*}
      and
    \begin{align}\label{diff_estimate_utUt}
          \begin{split}
            &\quad\big| p\|\nabla u\|^{p-2}(\nabla u,\nabla U)(u_t,U_t) \big| 
            \leq p\|\nabla u\|^{p-1}\|\nabla U\|\cdot |(u_t,U_t)| \\
            &\leq \frac{p}{4}\|\nabla u\|^{2p-2}\|\nabla U\|^2\|u_t\|^{2-p} + p\|u_t\|^{p-2}|(u_t,U_t)|^2 \\
            &\leq \frac{p}{4}(\|\nabla u\|^2+\|u_t\|^2)^{p/2}\|\nabla U\|^2 + p\|u_t\|^{p-2}|(u_t,U_t)|^2.
          \end{split}
    \end{align}
We derive from \eqref{diff_multi_E2p_1} that
    \begin{align}\label{diff_multi_E2p_2}
      \begin{split}
        &\frac{1}{2}\frac{d}{dt}(\|U_t\|^2 + \|\nabla U\|^2+f'(0)\|U\|^2)
            + \epsilon\frac{d}{dt}\Big(I_u^{\frac{p}{2}}(U_t,U)\Big)+ (1-\epsilon)I_u^{\frac{p}{2}}\|U_t\|^2\\&
            + \big(\epsilon-\frac{p}{4}\big)I_u^{\frac{p}{2}}\|\nabla U\|^2+\epsilon I_u^{\frac{p}{2}}f'(0)\|U\|^2 \leq - (f_*'(u)U,U_t)
            - \epsilon I_u^{\frac{p}{2}}I_{u,p}(U_t,U)\\&
            - \epsilon pI_u^{\frac{p}{2}}\|\nabla u\|^{p-2}(\nabla u,\nabla U)(u_t,U)- \epsilon pI_u^{\frac{p}{2}}\|u_t\|^{p-2}(u_t,U_t)(u_t,U)
            - \epsilon I_u^{\frac{p}{2}}(f_*'(u)U,U)\\&\quad
            + \epsilon C(\|U_t\|^2+\|\nabla U\|^2)I_u^{\frac{p}{2}-1}\left[I_{u,p}\|u_t\|^2 + |(f(u),u_t)|\right],
      \end{split}
    \end{align}
where we have also used the fact $I_u^{\frac{p}{2}}\leq I_{u, p}$. By the condition $f'(0)> -(1-\frac{p}{4})\lambda_1$, we can choose $\mu\in(0, 1-\frac{p}{4})$ such that $f'(0)>-\mu\lambda_1> -(1-\frac{p}{4})\lambda_1$ and deduce from Poincar\'{e}'s inequality
\begin{align*}
  \big(\epsilon-\frac{p}{4}\big)I_u^{\frac{p}{2}}\|\nabla U\|^2+\epsilon I_u^{\frac{p}{2}}f'(0)\|U\|^2>\big(\epsilon-\frac{p}{4}-\mu\epsilon\big)I_u^{\frac{p}{2}}\|\nabla U\|^2.
\end{align*}
This allows us to find $\epsilon\in(0, 1)$ and $c>0$ such that
\begin{align*}
  (1-\epsilon)I_u^{\frac{p}{2}}\|U_t\|^2+\big(\epsilon-\frac{p}{4}\big)I_u^{\frac{p}{2}}\|\nabla U\|^2+\epsilon I_u^{\frac{p}{2}}f'(0)\|U\|^2\geq 3cI_u^{\frac{p}{2}}(\|U_t\|^2+\|\nabla U\|^2).
\end{align*}
    Therefore, \eqref{diff_multi_E2p_2} gives
    \begin{align}\label{diff_multi_E2p_3}
      \begin{split}
            &\frac{1}{2}\frac{d}{dt}(\|U_t\|^2 + \|\nabla U\|^2+f'(0)\|U\|^2)
            + \epsilon\frac{d}{dt}\left(I_u^{\frac{p}{2}}(U_t,U)\right)
            + 3cI_u^{\frac{p}{2}}(\|U_t\|^2 + \|\nabla U\|^2)
            \\&\leq - (f_*'(u)U,U_t)
            - \epsilon I_u^{\frac{p}{2}}I_{u,p}(U_t,U)- \epsilon I_u^{\frac{p}{2}}\|\nabla u\|^{p-2}(\nabla u,\nabla U)(u_t,U)\\&\qquad
            - \epsilon p I_u^{\frac{p}{2}}\|u_t\|^{p-2}(u_t,U_t)(u_t,U)- \epsilon I_u^{\frac{p}{2}}(f_*'(u)U,U)\\&\qquad\qquad
            + \epsilon C(\|U_t\|^2+\|\nabla U\|^2)\left[I_{u,p}\|u_t\|^2 + |(f(u),u_t)|\right].
        \end{split}
    \end{align}

    It remains to estimate the right hand side of the above inequality. From \eqref{f_cond3} we know for any $\eta>0$ there exists $\delta>0$ such that
    \begin{equation*}
        |f_*'(s)|\leq
        \begin{cases}
            \eta|s|^p , &|s|<\delta, \\
            C(|s|+|s|^2)\leq C_{\delta}|s|^2, &|s|\geq \delta.
        \end{cases}
    \end{equation*}
Note that $\|\nabla u(t)\|$ and $\|u_t(t)\|$ keep small for all $t\geq 0$ if we take $r_0$ small enough by the continuity argument as in Lemma \ref{u_decay_equiv}. In this way, we can find $C_1>0$ such that if $r_0<C_1$, we have
    \begin{equation*}
        \begin{aligned}
            \|f_*'(u)\|_3 
            &\leq \Big(\int_{\{|u|<\delta\}}|f_*'(u)|^3dx + \int_{\{|u|\geq \delta\}} |f_*'(u)|^3 dx\Big)^{\frac{1}{3}} 
            \leq \eta\|u\|_{L^{3p}(\Omega)}^p + C\|u\|_{L^6(\Omega)}^2 \\&\leq C\eta\|\nabla u\|^p + C\|\nabla u\|^2 \leq c\|\nabla u\|^p,
        \end{aligned}
    \end{equation*}
if we choose $\eta$ small enough in advance, and thereby
    \begin{equation*}
        |(f_*'(u)U,U_t)| \leq \|f_*'(u)\|_3 \|U_t\| \|U\|_6 \leq cI_u^{\frac{p}{2}} (\|\nabla U\|^2 + \|U_t\|^2).
    \end{equation*}
    For the rest terms, one can estimate them more easily since they all contain higher-power factors $I_u^{\frac{q}{2}}$ with $q>p$. In detail, there exists $C_2>0$ such that if $r_0<C_2$, the rest of the right hand side of inequality \eqref{diff_multi_E2p_3} will be smaller than $cI_u^{\frac{p}{2}} (\|\nabla U\|^2 + \|U_t\|^2)$. Hence, we end up with
    \begin{equation}\label{E_Up_ineq}
        \frac{d}{dt}\left( \frac{1}{2}(\|U_t\|^2 + \|\nabla U\|^2+f'(0)\|U\|^2) + \epsilon I_u^{\frac{p}{2}}(U_t,U) \right) + cI_u^{\frac{p}{2}}(\|U_t\|^2 + \|\nabla U\|^2) \leq 0.
    \end{equation}
Note that $\frac{1}{2}(\|U_t\|^2 + \|\nabla U\|^2+f'(0)\|U\|^2) + \epsilon I_u^{\frac{p}{2}}(U_t,U)$ is equivalent to $\|U_t\|^2 + \|\nabla U\|^2$ if $I_u^{\frac{p}{2}}$ is small enough. Therefore, following the continuity argument as in Lemma \ref{u_decay_equiv} again, we can find a constant $C_3>0$ such that if $r_0<C_3$, then
    \begin{equation*}
        \begin{aligned}
            \|U_t(t)\|^2 + \|\nabla U(t)\|^2 
            &\leq C\Big[\frac{1}{2}(\|U_t\|^2 + \|\nabla U\|^2+f'(0)\|U\|^2) + \epsilon I_u^{\frac{p}{2}}(U_t,U)\Big]\Big|_{t} \\
            &\leq C\Big[\frac{1}{2}(\|U_t\|^2 + \|\nabla U\|^2+f'(0)\|U\|^2) + \epsilon I_u^{\frac{p}{2}}(U_t, U)\Big]\Big|_{t=0}\\
            & \leq C.
        \end{aligned}
    \end{equation*}
    This means that $L(t; \cdot)$ is uniformly bounded, i.e., $\|L(t; \cdot)\|_{X\to X}\leq C$ for all $t\geq 0$. 

    As the conclusion, for every $\omega_1,\omega_2\in B(0,r_0)$ and $t\geq 0$, we obtain
    \begin{equation*}
            \|S(t)\omega_1 - S(t)\omega_2\|_X 
            \leq \Big\|\int_0^1 L(t; \theta\omega_1 + (1-\theta)\omega_2) (\omega_1-\omega_2) d\theta \Big\|_X \leq C\|\omega_1-\omega_2\|_X.
    \end{equation*}
    This completes the proof. 
  \end{proof}

  \begin{remark}\label{remark_reason}
    Inequality \eqref{E_Up_ineq} implies that
    \begin{equation*}
      \frac{d}{dt}\Lambda_U(t) +c I_u^{\frac{p}{2}}(t)\Lambda_U(t) \leq 0,
    \end{equation*}
    where $\Lambda_U=\frac{1}{2}(\|U_t\|^2 + \|\nabla U\|^2+f'(0)\|U\|^2) + \epsilon I_u^{\frac{p}{2}}(U_t,U)$. By the Gronwall lemma, we get 
        \begin{equation*}
            \Lambda_U(t) \leq \Lambda_U(0)\exp\Big(\intt -cI_u^{\frac{p}{2}}(s) ds\Big).
        \end{equation*}
Then it follows from Lemma {\rm \ref{u_decay_equiv}} that
        \begin{equation*}
            \begin{aligned}
                \Lambda_U(t)&\leq \Lambda_U(0) \exp\left(C_1^p\ln(k_1(I_u(0))^{-p/2})-C_1^p\ln(t+k_1(I_u(0))^{-p/2})\right)  \\
                &= \Lambda_U(0) \left(\frac{k_1(I_u(0))^{-p/2}}{t+k_1(I_u(0))^{-p/2}}\right)^{C_1^p},
            \end{aligned}
        \end{equation*}
        where $C_1, k_1$ are the constants given in Lemma {\rm \ref{u_decay_equiv}}. Note that $\Lambda_U^{1/2}$ can be regarded as an equivalent quantity of $\|(U, U_t)\|_{H^1_0(\Omega)\times L^2(\Omega)}$ when $I_u$ is small enough. This tells us that, for each $(u_0,u_1)\in B(0,r_0)$, the derivative $L(t;u_0,u_1)$ would have norm less than 1, in fact tending to zero, as long as $t$ is large enough. However, this is not the case if we consider the supremum of the operator norm of $L(t;u_0,u_1)$ over $(u_0,u_1)\in B(0,r_0)$, since for fixed $T>0$
        \begin{equation*}
            \frac{k_1(I_u(0))^{-p/2}}{T+k_1(I_u(0))^{-p/2}} \to 1 \quad \mathrm{as}~I_u(0)\to 0.
        \end{equation*}
        In fact, if $(u_0,u_1)=(0,0)$, the linearized equation \eqref{diff_eq} turns exactly into the wave equation
        \begin{align*}
          U_{tt}-\Delta U+f'(0)U=0,
        \end{align*}
which conserves the quantity $\Lambda_U$. Hence, the linearized operator has no any uniform contraction property and the known methods to estimate the fractal dimension do not work near the origin.
    \end{remark}

    To show finite dimensionality of the non-degenerate part of the attractor, we introduce the decomposition $U=V+W$, where
    \begin{equation}\label{V_eq}
        \begin{cases}
            V_{tt} - \Delta V + \effip V_t = 0, \\
            (V(0),V_t(0)) = (\xi,\zeta),
        \end{cases}
    \end{equation}
    and
    \begin{equation}\label{W_eq}
        \begin{cases}
            W_{tt} - \Delta W + \effip W_t + p[\|\nabla u\|^{p-2}(\nabla u,\nabla U) \\
            \qquad\qquad\qquad\qquad\qquad\qquad\qquad\,\, +\|u_t\|^{p-2}(u_t,U_t)]u_t + f'(u)U=0, \\
            W(0) = W_t(0) = 0.
          \end{cases}
    \end{equation}

  \begin{lemma}\label{V_decay}
       Suppose Assumptions \ref{f_assum1} and \ref{assum_2} hold. Then there exist $\epsilon_0>0, T_0>0$ and $0<q<1$ such that if $(u, u_t)\in \mathscr{A}$ with $\|\nabla u(0)\|^2+\|u_t(0)\|^2 \geq \epsilon_0^2$, then the solution of \eqref{V_eq} satisfies
    \begin{equation*}
        I_V(T)\leq qI_V(0), \quad\forall T\geq T_0,
    \end{equation*}
      where recall $I_V=\|\nabla V\|^2+\|V_t\|^2$. 
  \end{lemma}
    \begin{proof}
Let $(u, u_t)\in \mathscr{A}$ and $I_u(0)\geq \epsilon_0^2$ with $\epsilon_0$ to be determined. We multiply \eqref{V_eq} by $V_t + \epsilon V$ with $\epsilon>0$ and integrate over $\Omega$ to obtain
    \begin{equation*}
        \begin{aligned}
            \frac{d}{dt}I_{V, \epsilon}(t) + (\|\nabla u(t)\|^p 
            &+ \|u_t(t)\|^p - \epsilon)\|V_t(t)\|^2 + \epsilon \|\nabla V(t)\|^2 \\
            &+ \epsilon\effitp(V(t),V_t(t)) = 0,
        \end{aligned}
    \end{equation*}
    where $I_{V, {\epsilon}}(t) = \frac{1}{2}\|V_t(t)\|^2 + \frac{1}{2}\|\nabla V(t)\|^2 + \epsilon(V(t),V_t(t))$. Note that
    \begin{equation*}
            \left|\epsilon\effip(V,V_t)\right| \leq \epsilon C(V, V_t)\leq \frac{\epsilon}{2}\|\nabla V\|^2 + \epsilon C\|V_t\|^2.
    \end{equation*}
In terms of Remark \ref{regul-decay_remark_decay} and \eqref{E_equiv} we have that
    \begin{equation*}
        \frac{d}{dt}I_{V, {\epsilon}}(t) + \Big(\frac{C_1}{t+k_1\epsilon_0^{-p}} -C_0\epsilon\Big)\|V_t(t)\|^2 + \frac{\epsilon}{2} \|\nabla V(t)\|^2 \leq 0.
    \end{equation*}
    Set $T_0=k_1\epsilon_0^{-p}$. Then, for $T\geq T_0$, it holds for $t\in [0, T]$
    \begin{align*}
      \frac{d}{dt}I_{V, {\epsilon}}(t) + \Big(\frac{C_1}{T+k_1\epsilon_0^{-p}} -C_0\epsilon\Big)\|V_t(t)\|^2 + \frac{\epsilon}{2} \|\nabla V(t)\|^2 \leq 0,\\
      \frac{d}{dt}I_{V, {\epsilon}}(t) + \frac{C_1}{4C_0(T+k_1\epsilon_0^{-p})}(\|V_t(t)\|^2 + \|\nabla V(t)\|^2) \leq 0,
      \end{align*}
by setting $\epsilon =\frac{C_1}{2C_0(T+k_1\epsilon_0^{-p})}$. Since
    \begin{equation*}
        \frac{1-\epsilon/\sqrt{\lambda_1}}{2}I_V(t) \leq I_{V, {\epsilon}}(t) \leq \frac{1+\epsilon/\sqrt{\lambda_1}}{2}I_V(t)\quad\textrm{for } \epsilon<\sqrt{\lambda_1},
    \end{equation*}
we have for $t\in [0, T]$
    \begin{equation*}
        \frac{d}{dt}I_{V, {\epsilon}}(t) + \frac{C}{(1+\epsilon/\sqrt{\lambda_1})T} I_{V, {\epsilon}}(t) \leq 0.
    \end{equation*}
    Hence, it follows that
    \begin{equation*}
        \begin{aligned}
            I_V(T)&\leq \frac{2}{1-\epsilon/\sqrt{\lambda_1}}I_{V, {\epsilon}}(T) 
            \leq \frac{2}{1-\epsilon/\sqrt{\lambda_1}}I_{V, {\epsilon}}(0) e^{-\frac{C}{(1+\epsilon/\sqrt{\lambda_1})T}\times T} \\
            &\leq \frac{1+\epsilon/\sqrt{\lambda_1}}{1-\epsilon/\sqrt{\lambda_1}} e^{-\frac{C}{2}}I_V(0).
        \end{aligned}
    \end{equation*}
    Observe that $$\frac{1+\epsilon/\sqrt{\lambda_1}}{1-\epsilon/\sqrt{\lambda_1}}e^{-\frac{C}{2}}\leq \frac{4\sqrt{\lambda_1}C_0k_1\epsilon_0^{-p}+C_1}{4\sqrt{\lambda_1}C_0k_1\epsilon_0^{-p}-C_1}e^{-\frac{C}{2}}\stackrel{\Delta}{=}q $$ and $\lim_{\varepsilon_0\rightarrow 0^+} q=e^{-\frac{C}{2}}<1$. Therefore, we can take $\epsilon_0$ small enough, such that, $$\epsilon_0<r_0, \epsilon<\sqrt{\lambda_1} \textrm{ and } q<1,$$in which situation the desired result holds. 
  \end{proof}

    \begin{lemma}\label{W_regularity}
        Suppose Assumptions \ref{f_assum1} and \ref{assum_2} hold. Let $(u,u_t)\in\mathscr{A}$ and $\|(\xi, \zeta)\|_{H^1_0\times L^2} \leq 1$. Then, for each $T>0$, $(\nabla W(T),W_t(T))\in H^{\beta+1}(\Omega)\times H^{\beta}(\Omega)$ for some $\beta>0$.
    \end{lemma}
    \begin{proof}
      As in the proof of Lemma \ref{w_regularity}, it suffices to show the higher regularity of $f'(u)U$ and $[\|\nabla u\|^{p-2}(\nabla u,\nabla U)+\|u_t\|^{p-2}(u_t,U_t)]u_t$. Note that $(u, u_t)$ is uniformly bounded in $H^{1+\frac{2}{7}}(\Omega)\times H^{\frac{2}{7}}(\Omega)$ by Theorem \ref{attractor_reg}. For $\beta_1=\frac{4}{7}$ and $q=\frac{14}{9}$
      \begin{align*}
        \|f'(u) U\|_{H^{\beta_1}}\leq C\|f'(u) U\|_{W^{1, q}}\leq \|f''(u)\nabla u U\|_{L^q}+\|f'(u)\nabla U\|_{L^q}\leq C\|\nabla U\|,
      \end{align*}
   and for $\beta_2=\frac{2}{7}$
    \begin{align*}
        \|[\|\nabla u\|^{p-2}(\nabla u,\nabla U) + \|u_t\|^{p-2}(u_t,U_t)]u_t\|_{H^{\beta_2}} \leq C(\|\nabla U\|+\|U_t\|).
    \end{align*}
Therefore, for $\beta=\frac{2}{7}$ it holds uniformly in $t$ that
\begin{align*}
  \|f'(u)U\|_{H^\beta}+\|[\|\nabla u\|^{p-2}(\nabla u,\nabla U)+\|u_t\|^{p-2}(u_t,U_t)]u_t\|_{H^\beta}\leq C(\|\nabla\xi\|+\|\zeta\|)\leq C, 
\end{align*}
if $\|(\xi, \zeta)\|_{\VV\times\HH}\leq 1$. Following the similar lines as in the proof of Lemma \ref{w_regularity}, we complete the proof. 
\end{proof}

Now we are in position to show the finite dimensionality of the global attractor. 
    \begin{prop}[\cite{MR3408002}]\label{dim_chueshov}
        Let $M$ be a compact set in a Banach space $X$ and $V:M\mapsto X$ be uniformly quasi-differentiable on $M$. Assume that the quasi-derivative $L(u)$ can be split into two linear parts
        \begin{equation*}
            L(u) = L^1(u) + L^2(u), \quad u\in M,
        \end{equation*}
        where
        \begin{equation*}
            \sup_{u\in M} \|L^1(u)\| \equiv q<1
        \end{equation*}
        and $L^2(u)$ is a compact operator on $X$ for each $u\in M$. We also assume that the function $u\mapsto L^2(u)$ is continuous in the operator norm. If $M\subset VM$, then $d_B(M)$ is finite.
    \end{prop}
    \begin{thm}\label{dim_nondegenerate}
        Suppose Assumptions \ref{f_assum1} and \ref{assum_2} hold. There exists $\epsilon_0>0$ such that for any $\epsilon\in(0,\epsilon_0)$, $\mathscr{A}\backslash B(0,\epsilon)$ has finite fractal dimension in $\VV\times\HH$.
    \end{thm}
    \begin{proof}
    We have already known that $S(t)\omega_0$ decays uniformly near $0$ by Lemma \ref{u_decay_equiv}, which also indicates that $S(t)$ is negatively invariant on $\mathscr{A}\backslash B(0,\epsilon)$ for small $\epsilon$ and large $t$ (depending on $\epsilon$). Due to the analysis above, we have the decomposition
    \begin{equation*}
        \begin{aligned}
            L(t;\omega_0)(\xi,\zeta) = (U(t),U_t(t)) 
            &= (V(t),V_t(t)) + (W(t),W_t(t)) \\
            &= L_1(t;\omega_0)(\xi,\zeta) + L_2(t;\omega_0)(\xi,\zeta),
        \end{aligned}
      \end{equation*}
      where $\|L_1(t;\omega_0)\|<1$ for $t> T_0$ by Lemma \ref{V_decay}, and $L_2(t;\omega_0)$ is compact in $\VV\times\HH$ for each $t>0$ by Lemma \ref{W_regularity}. Moreover, it is easy to verify that $L^2(t;\omega_0)$ is continuous in $\omega_0$. By virtue of Proposition \ref{dim_chueshov} we complete the proof by choosing $t>T_0$ large enough. 
  \end{proof}

  \begin{proof}[Proof of Theorem \ref{sec_introduction_thm_main}]
      Note that $0\in \mathscr{A}$. According to the proof of Lemma \ref{u_decay_equiv}, it holds in $\mathscr{A}\cap B(0, r_0)$ that
      \begin{equation*}
        \frac{d}{dt}\left[E(t) + \frac{\epsilon}{2}E(t)^{\frac{p}{2}}(u_t,u)\right] + C\left[E(t) + \frac{\epsilon}{2}E(t)^{\frac{p}{2}}(u_t,u)\right]^{\frac{p}{2}+1} \leq 0,
      \end{equation*}
      and $E(t) + \frac{\epsilon}{2}E(t)^{\frac{p}{2}}(u_t,u)$ is equivalent to $I_u$ in $\mathscr{A}\cap B(0, r_0)$. 
Combining these facts with Corollary \ref{Holder_t}, Lemma \ref{Lipschitz_u} as well as Theorem \ref{dim_nondegenerate}, we complete the proof by applying Theorem \ref{main_2}. 
  \end{proof}
  
\section*{Acknowledgements}
This work is supported by the National Natural Science Foundation of China (11731005, 12201604, 12371106) and the Fundamental Research Funds for the Central Universities (E3E40102X2). 
  
\section*{Data Availability Statement}
Data sharing is not applicable to this article as no new data were created or analyzed in this study. 

\bibliographystyle{plain}
\bibliography{reference.bib}

\vspace{0.5cm}

\noindent Zhijun Tang\\
Department of Mathematics, Nanjing University, Nanjing, 210093, China\\
E-mail: tzj960629@163.com
\vspace{0.3cm}

\noindent Senlin Yan\\
Department of Mathematics, Nanjing University, Nanjing, 210093, China\\
E-mail: dg20210019@smail.nju.edu.cn
\vspace{0.3cm}

\noindent Yao Xu\\
School of Mathematical Sciences, University of Chinese Academy of Sciences, Beijing, 100049, China\\
E-mail: xuyao89@gmail.com, xuyao@ucas.ac.cn
\vspace{0.3cm}

\noindent Chengkui Zhong\\
Department of Mathematics, Nanjing University, Nanjing, 210093, China\\
E-mail: ckzhong@nju.edu.cn

\end{document}